\documentclass[12pt]{article}
\usepackage{amsmath,amssymb,amsthm,mathtools,mathrsfs,latexsym}
\usepackage{graphicx}
\usepackage{xcolor}
\usepackage{tikz}
\usepackage[linkcolor=blue,colorlinks=true,urlcolor=red,bookmarksopen=true]{hyperref}
\usepackage{newclude}
\usepackage{imakeidx}
\makeindex

\allowdisplaybreaks

 \numberwithin{equation}{section}

\newtheorem{thm}{Theorem}[section]

\newtheorem{lema}{Lemma}[section]
\newtheorem{prop}{Proposition}[section]

\newtheorem{rmk}{Remark}[section]

\renewcommand{\Pi}{\mathcal{P}}

\renewcommand{\Xi}{\mathcal{X}}

\let\div\relax
\DeclareMathOperator*{\div}{div}
\DeclareMathOperator*{\curl}{curl}

\DeclarePairedDelimiter\abs{\lvert}{\rvert}
\DeclarePairedDelimiter\norm{\lVert}{\rVert}
\makeatletter
\let\oldabs\abs
\def\abs{\@ifstar{\oldabs}{\oldabs*}}
\let\oldnorm\norm
\def\norm{\@ifstar{\oldnorm}{\oldnorm*}}
\makeatother

\usepackage{scalerel}
\usepackage[usestackEOL]{stackengine}
\def\dashint{\,\ThisStyle{\ensurestackMath{
\stackinset{c}{.2\LMpt}{c}{.5\LMpt}{\SavedStyle-}{\SavedStyle\phantom{\int}}
}
\setbox0=\hbox{$\SavedStyle\int\,$}\kern-\wd0}\int}

\renewcommand{\em}[1]{\textbf{#1}\index{#1}}

\newcommand\blfootnote[1]{%
  \begingroup
  \renewcommand\thefootnote{}\footnote{#1}%
  \addtocounter{footnote}{-1}%
  \endgroup
}

\usepackage{adjustbox}
\usepackage{wrapfig}
\usepackage{enumitem}
\usepackage{protosem}

\title{Global well-posedness for 2D compressible radially symmetric Navier-Stokes equations with swirl}
\date{}
\author{
\bf\large Xiangdi Huang$^a$, Weili Meng$^a$*\\
\small a. State Key Laboratory of Mathematical Sciences, Academy of Mathematics and Systems Science,\\
\small Chinese Academy of Sciences, Beijing 100190, China;
\blfootnote{Email addresses: xdhuang@amss.ac.cn\,\, (X. Huang); mengweili@amss.ac.cn\,\, (W. Meng). } }

\begin{document}
\maketitle
\begin{abstract}
In this paper, we consider the radially symmetric compressible Navier-Stokes equations with swirl in two-dimensional disks, where the shear viscosity coefficient \(\mu = \text{const}> 0\), and the bulk one \(\lambda = \rho^\beta(\beta>0)\). When \(\beta \geq 1\), we prove the global existence and asymptotic behavior of the large strong solutions for initial values that allow for vacuum. One of the key ingredients is to show the uniform boundedness of the density independent of the time. When \(\beta\in(0,1)\), we prove the same conclusion holds when the initial value satisfies \(\norm{\rho_0}_{L^\infty} \leq a_0\), where \(a_0\) is given by \eqref{def a_0} as in Theorem \ref{Thm3}. To the best of our knowledge, this is the first result on the global existence of large strong solutions for 2D compressible Navier-Stokes equation with real non-slip (non Navier-slip) boundary conditions when $\beta\ge1$ and the first result on the global existence of  strong solutions when $\beta\in(0,1)$.\\[4mm]
\end{abstract}

\section{Introduction}
The paper is concerned with the compressible Navier-Stokes equations in a bounded domain $\Omega\subset \mathbb{R}^2$.
\begin{equation}
\label{Equ1}
	\begin{dcases}
		\rho_t+\div(\rho \mathbf{u})=0,\\
(\rho \mathbf{u})_t+\div (\rho \mathbf{u}\otimes \mathbf{u})+\nabla P=\mu \Delta \mathbf{u}+\nabla((\mu+\lambda)\div \mathbf{u}),\\
	\end{dcases}
\end{equation}
where $\rho=\rho(x,t), \mathbf{u}=(u_1(x,t), u_2(x,t))^T$ represent the unknown density field and velocity field respectively. The pressure $P$ has the following form
\begin{align*}
P=A\rho^\gamma,
\end{align*}
where the constants $A>0$ and $\gamma>1$. Assume that the shear viscosity coefficient and the bulk viscosity coefficient satisfy
\begin{align*}
0<\mu=\text{const},\quad \lambda(\rho)=b\rho^\beta,
\end{align*}
the constants $\beta$ and $b$ are positive. Without loss of generality, we always assume
\begin{align*}
A=b=1.
\end{align*}
We consider the equation \eqref{Equ1} with the following initial data
\begin{align}\label{inital data1}
\rho(x,0)=\rho_0(x),\quad \mathbf{u}(x,0)=\mathbf{u_0}(x), \quad x\in\Omega,
\end{align}
and Dirichlet boundary condition
\begin{align}\label{boundary condition1}
\mathbf{u}(x,t)=0,\quad x\in \partial \Omega.
\end{align}

For the compressible Navier-Stokes equations with constant viscosity coefficients, there have been many significant research advancements. The local existence and uniqueness of classical solutions away from vacuum were established by Nash \cite{1962 Nash-BSMF} and Serrin \cite{1959 Serrin-ARMA}. For initial values allowing for vacuum, Cho-Choe-Kim \cite{2004 Cho-Choe-Kim-JMPA} and Choe-Kim \cite{2003 Choe-Kim-JDE} established the local existence and uniqueness of strong solutions. For global classical solutions, Matsumura-Nishida \cite{1980 Matsumura-Nishida-JMKU} established the global existence of classical solutions for initial values close to a non-vacuum equilibrium state in some Sobolev space \(H^s\). Later, Hoff \cite{1995 Hoff-JDE,1995 Hoff-ARMA} studied the problem for discontinuous initial values. For the case of arbitrarily large initial values, the most significant breakthrough came from Lions \cite{1998-Lions}. He proved the global existence of finite-energy weak solutions in three dimensions when \(\gamma \geq \frac{9}{5}\). Later, Jiang-Zhang \cite{2001 Jiang-Zhang-CMP} relaxed the restriction on $\gamma$ to $\gamma>1$ under the assumption of spherical symmetry. Feireisl-Novotn\'{y}-Petzeltov\'{a} \cite{2001 Feireisl-Novotny-Petzeltova-JMFM} overcame the difficulties associated with renormalized solutions by introducing oscillating defect measures, thereby relaxing the restriction to \(\gamma > \frac{3}{2}\). Recently, Huang-Li-Xin \cite{2012 Huang-Li-Xin-CPAM} and Li-Xin \cite{2019 Li-Xin-AnnPDE} established the existence and uniqueness of global classical solutions for 3D and 2D respectively, assuming only that the initial energy is small while allowing for initial vacuum and large oscillations.

For the viscosity coefficients \(\mu = \text{const}\) and \(\lambda = \rho^\beta\), there is an abundance of research literature. Vaigant-Kazhikhov \cite{1995 Vaigant-Kazhikhov-SMJ} first established the existence and uniqueness of strong solutions in rectangle domains under the condition \(\beta > 3\) and initial density away from vacuum. Later, Jiu-Wang-Xin \cite{2014 Jiu-Wang-Xin-JMFM} established the global existence of classical solutions that allow for vacuum initial density on the \( \mathbb{T}^2 \) domain, under the condition \(\beta > 3\). Huang-Li \cite{2016 Huang-Li-JMPA}, by applying some new ideas on commutator estimates, relaxed the restriction on the viscosity coefficient to \(\beta > \frac{4}{3}\) and removed the assumption that initial density is away from vacuum. Moreover, they also demonstrated that under the conditions \(\gamma > \frac{3}{2}\) and \(1 < \gamma < 4\beta - 3\), the solutions converge to an equilibrium state as \(T \to \infty\). For the Cauchy problem, Huang-Li \cite{2022 Huang-Li-JMA} proved the global existence and uniqueness of both strong and classical solutions under the condition \(\beta > \frac{4}{3}\) and with the initial density \(\rho_0\) having a certain rate of decay (see also Jiu-Wang-Xin \cite{2013 Jiu-Wang-Xin-JDE}). One can also refer to Jiu-Wang-Xin \cite{2018 Jiu-Wang-Xin-Phys} for positive density at far field. For the free boundary problem, Li-Zhang \cite{2016 Li-Zhang-JDE} established the existence and uniqueness of strong solutions under the condition \(\beta > 1\), and they also obtained estimates for the rate of boundary expansion under the condition \(1 < \beta \leq \gamma\). Huang-Meng-Ni \cite{2025 Huang-Meng-Ni-JMAA} utilized the geometric properties of 2D radial symmetry to extend the condition to $\beta=1$. For the Navier-slip boundary, Fan-Li-Li \cite{2022 Fan-Li-Li-ARMA}, using the Riemann mapping theorem and the pull-back Green’s function, established the existence and uniqueness of strong solutions in general bounded simply connected domains under the restriction \(\beta > \frac{4}{3}\). Later, Fan-Li-Wang \cite{2023 Fan-Wang-Li-arXiv} fully utilized the damping mechanism of the pressure and demonstrated that under the same conditions, the global strong solutions converge to equilibrium at an exponential rate as $T\rightarrow\infty$. Moreover, their results also hold for periodic domains, which improve upon the asymptotic behavior results of Huang-Li \cite{2016 Huang-Li-JMPA}. For the Dirichlet problem, under the assumption of radial symmetry, Huang-Su-Yan-Yu \cite{2023 Huang-Su-Yan-Yu} provided a representation of the effective viscous flux boundary term and demonstrated the existence and uniqueness of global strong solutions under the condition $\beta>1$. Furthermore, they showed that under the condition $\max\left\{1,\frac{\gamma+2}{4}\right\}<\beta\leq \gamma$, the solutions converge to an equilibrium state as $T\rightarrow \infty$. Later, Huang-Yan \cite{2023 Huang-Yan} removed the restriction \(\beta \leq \gamma\) when considering the asymptotic behavior of strong solutions in the MHD system. 

In this paper, we consider the global existence and asymptotic behavior of 2D radially symmetric strong solutions with swirl. Under our settings, the domain, density field, and velocity field have the following radially symmetric representations:
\begin{align*}
\Omega=B_R,\quad \rho(x,t)=\rho(r,t),\quad \mathbf{u}(x,t)=u_r(r,t)\mathbf{e_r}+u_\theta(r,t)\mathbf{e_\theta},
\end{align*}
where $\mathbf{e_r}=\left(\frac{x_1}{r},\frac{x_2}{r}\right)$ and $\mathbf{e_\theta}=\left(\frac{x_2}{r},-\frac{x_1}{r}\right)$ represent the radial vector and the angular vector respectively. The initial boundary value problem \eqref{Equ1}-\eqref{boundary condition1} is transformed into 
\begin{align}\label{Equ 2}
\left\{ 
\begin{array}{ll}
\rho_t+\partial_r(\rho u_r)+\frac{\rho u_r}{r}=0,\\
(\rho u_r)_t+\partial_r(\rho u_r^2)+\frac{\rho}{r}(u^2_r-u_\theta^2)=\partial_r((2\mu+\lambda)(\partial_r u_r+\frac{u_r}{r})-P),\\
(\rho u_\theta)_t+\partial_r(\rho u_ru_\theta)+\frac{2\rho}{r}u_ru_\theta=\mu(\partial_{rr}u_\theta+\frac{\partial_r u_\theta}{r}-\frac{u_\theta}{r^2}).
\end{array}
\right.
\end{align}
with the initial data
\begin{align}\label{inital data2}
\rho(r,0)=\rho_0(r),\quad u_r(r,0)=u_{r_0}(r),\quad u_\theta(r,0)=u_{\theta_0}(r),
\end{align}
and the Dirichlet boundary condition
\begin{align}\label{boundary condition2}
u_r(0,t)=u_r(R,t)=u_\theta(0,t)=u_\theta(R,t)=0.
\end{align}

Before stating our main results, we first clarify the following notations. For a positive integer $k$ and $p\ge 1$, we denote the standard Lebesgue and Sobolev spaces as follows:
\begin{align*}
L^p=L^p(\Omega),\quad W^{k,p}=W^{k,p}(\Omega),\quad H^k=W^{k,2}(\Omega).
\end{align*}
For a scalar function \(f\), define the material derivative of \(f\) as
\begin{align*}
\dot{f}=\frac{D}{Dt}f=\partial_t f+\mathbf{u}\cdot \nabla f,
\end{align*}
and the average value of the integral as
\begin{align*}
\bar{f}=\frac{1}{|\Omega|}\int_\Omega fdx.
\end{align*}
For the vector-valued function \(\mathbf{u}\), define
\begin{align*}
\div \mathbf{u}=\partial_1 u_1+\partial_2 u_2,\quad w=\nabla^\perp\cdot \mathbf{u}=(\partial_2,-\partial_1)\cdot(u_1,u_2)=\partial_2 u_1-\partial_1 u_2.
\end{align*}

Now we state the first result concerning the global existence and asymptotic behavior of the strong solutions for large initial data when $\beta\ge 1$.

\begin{thm}\label{Thm1}
Suppose that 
\begin{align*}
\beta\ge 1, \quad \gamma>1,
\end{align*}
and the initial data $(\rho_0, \mathbf{u_0})$ satisfy
\begin{align*}
0\le\rho_0\in W^{1,q},\quad \mathbf{u_0} \in H_0^1,
\end{align*}
for some $q>2$, and satisfy the following compatibility condition
\begin{align*}
-\mu\Delta \mathbf{u_0}-\nabla((\mu+\lambda(\rho_0))\div\mathbf{u_0})+\nabla P(\rho_0)=\sqrt{\rho_0}g,
\end{align*}
for some $g\in L^2$. Then the problem \eqref{Equ 2}-\eqref{boundary condition2} has a unique strong solution satisfying 
\begin{align*}
\rho\in C([0,T];W^{1,q}), \quad \mathbf{u} \in C([0,T]; H^1),
\end{align*}
and 
\begin{align}\label{regularity}
\left\{ 
\begin{array}{lc}
\norm{\sqrt{\rho}\mathbf{u}}_{L^\infty L^2}+\norm{\rho}_{L^\infty L^\gamma}+\norm{\nabla \mathbf{u}}_{L^2L^2}\leq C,\\
\norm{\sqrt{\rho}\dot{\mathbf{u}}}_{L^\infty L^2}+\norm{\nabla \dot{\mathbf{u}}}_{L^2L^2}\leq C.
\end{array}
\right.
\end{align}
for any $0<T<\infty$. Moreover, there exists a constant $C>0$ depending on $\mu,\beta,\gamma,\Omega,\norm{\rho_0}_{L^\infty}$ and $\norm{\mathbf{u_0}}_{H^1}$ but not on time $T$ such that
\begin{align*}
\sup_{0\leq t<\infty}\norm{\rho(t)}_{L^\infty}\leq C,
\end{align*}
and the following asymptotic behaviors hold
\begin{align}\label{long time behaviors}
\lim\limits_{t\rightarrow\infty}(\norm{\rho-\rho_s}_{L^p}+\norm{\nabla\mathbf{u}}_{L^p})=0
\end{align}
for any $1\le p<\infty$, where 
\begin{align*}
\rho_s=\frac{1}{|\Omega|}\int_\Omega\rho_0dx.
\end{align*}
\end{thm}
Theorem 1.1 is a comprehensive improvement on existing results, mainly reflected in the following 4 remarks.
\begin{rmk}
In their paper \cite{2023 Huang-Su-Yan-Yu}, Huang-Su-Yan-Yu proved the global existence of strong solutions to the problem \eqref{Equ 2}-\eqref{boundary condition2} without swirl when $\beta>1$. To be precise, $u_\theta$ is equal to $0$. As for the endpoint case $\beta = 1$, Huang-Meng-Ni \cite{2025 Huang-Meng-Ni-JMAA} considered the global existence of strong solutions for the free boundary problem.
The global existence in Theorem \ref{Thm1} can be regarded as an extension of the results in \cite{2023 Huang-Su-Yan-Yu} to the endpoint case $\beta=1$ and the case with swirl.
\end{rmk}
\begin{rmk}
It is worth noting that, in the absence of swirl, the Dirichlet boundary condition is essentially a Navier-slip boundary condition, so the estimates of the effective viscous flux and vorticity can be obtained. However, in the presence of swirl, the Dirichlet boundary condition is no longer a Navier-slip boundary condition. Therefore, how to obtain the estimates of the effective viscous flux and vorticity in the absence of boundary conditions is the main obstacle we face.
\end{rmk}
\begin{rmk}
In \cite{2023 Huang-Su-Yan-Yu}, they obtained the asymptotic behavior $\eqref{long time behaviors}$  under the condition $\gamma>1, \max\left\{1,\frac{\gamma+2}{4}\right\}<\beta\leq \gamma$. In \cite{2023 Huang-Yan}, Huang-Yan removed the restriction $\beta\leq \gamma$ when considering the global existence and asymptotic behavior of strong solutions of the MHD system. When \(\beta > 1\), we remove the restriction $\beta>\frac{\gamma+2}{4}$ by using a key idea from \cite{2023 Fan-Wang-Li-arXiv}. Specifically, we also established the asymptotic behavior of the solution at the endpoint case \(\beta = 1\), which is essentially different from the case where \(\beta > 1\).
\end{rmk}
\begin{rmk}
The result also holds for the compressible radially symmetric Navier-Stokes equation without swirl.
\end{rmk}

Now we introduce our second result, the global existence and asymptotic behavior of strong solutions with small initial density when $0<\beta<1$. We need the following auxiliary functions:
{
\small
\begin{align*}
&\theta(a)=\theta_{\mu,\beta}(a)=2\mu\log a+\frac{1}{\beta}a^\beta,\quad\tilde{M}(a)=\tilde{M}_{R}(a)=\pi R^2 a,\\
&\tilde{E}(a)=\tilde{E}_{\gamma, R, \norm{\nabla \mathbf{u}_0}_{L^2}}(a)=\frac{\pi R^2}{\gamma-1}a^\gamma+\frac{R^2}{4}\norm{\nabla \mathbf{u}_0}_{L^2}^2a,\quad
\tilde{U}(a)=\tilde{U}_{R, \norm{\nabla \mathbf{u}_0}_{L^2}}(a)=\frac{\pi R^2}{(2\pi)^{\frac{3}{2}}}\norm{\nabla\mathbf{u}_0}_{L^2}^3 a,\\
&K(a)=K_{\mu, \beta,\gamma, R, \norm{\nabla \mathbf{u}_0}_{L^2}}(a)=\max\left\{2\mu\log a+\frac{1}{\beta}a^\beta,2\mu\log((\frac{(\gamma-1)\tilde{E}(a)}{\pi R^2})^{\frac{1}{\gamma}})+\frac{1}{\beta}(\frac{(\gamma-1)\tilde{E}(a)}{\pi R^2})^{\frac{\beta}{\gamma}}\right\}\\
&\qquad+2^{\frac{4}{3}}R^{\frac{1}{3}}\pi^{-\frac{1}{3}}(7\mu)^{\frac{2}{3\beta}}\left(\tilde{U}(a)+3\sqrt{2}7^{\frac{\gamma}{\beta}}\mu^{\frac{\gamma}{\beta}-1} R\tilde{E}(a)\right)^{\frac{1}{3}}\\
&\qquad+\frac{1}{2 \mu\pi}(7\mu)^{\frac{1}{\beta}}\tilde{E}(a)+\frac{1}{2\pi R}(\tilde{M}(a)+2\tilde{E}(a))+\frac{2}{(\pi R^2)^\beta(1-\beta)}\tilde{M}(a)^\beta+\frac{\tilde{M}(a) \tilde{E}(a)}{4\mu\pi^2 R^2}.
\end{align*}
}
The result is stated as follows:
\begin{thm}\label{Thm3}
Suppose that 
\begin{align*}
0<\beta<1, \quad \gamma>1,
\end{align*}
and the initial data $(\rho_0, \mathbf{u_0})$ satisfy
\begin{align*}
0\le\rho_0\in W^{1,q},\quad \mathbf{u_0} \in H_0^1,
\end{align*}
for some $q>2$, and satisfy the following compatibility condition
\begin{align*}
-\mu\Delta \mathbf{u_0}-\nabla((\mu+\lambda(\rho_0))\div\mathbf{u_0})+\nabla P(\rho_0)=\sqrt{\rho_0}g,
\end{align*}
for some $g\in L^2$. Let \(a_0\) satisfy
\begin{align}\label{def a_0}
K(a_0)=\theta\left(\frac{99}{100}(7\mu)^{\frac{1}{\beta}}\right).
\end{align}
If
\begin{align}\label{rho 0 small}
\norm{\rho_0}_{L^\infty}\leq a_0 ,
\end{align}
then the problem \eqref{Equ 2}-\eqref{boundary condition2} has a unique strong solution satisfying 
\begin{align*}
\rho\in C([0,T];W^{1,q}), \quad \mathbf{u} \in C([0,T]; H^1),\quad \sup_{0\leq t\leq T}\norm{\rho(t)}_{L^\infty}\leq (7\mu)^{\frac{1}{\beta}},
\end{align*}
the regularity \eqref{regularity} and the asymptotic behaviors \eqref{long time behaviors} hold.
\end{thm}
\begin{rmk}
We note that \(K\) is a strictly monotonically increasing function defined on \((0, \infty)\), and \(K(0+)=-\infty\). Therefore, \(a_0\) is well-defined.
\end{rmk}
\begin{rmk}
Once the parameters \(\mu\), \(\beta\), \(\gamma\), \(R\) and the initial value \(\mathbf{u}_0\) are given, we can explicitly compute \(a_0\) by \eqref{def a_0}.
\end{rmk}

We now comment on the analysis of this paper. The previous literature \cite{2016 Huang-Li-JMPA}, \cite{2022 Fan-Li-Li-ARMA} and \cite{2023 Huang-Su-Yan-Yu} shows that the strong solutions of the system $\eqref{Equ1}$-$\eqref{boundary condition1}$ blow up in finite time, typically due to the blow up of the \(L^\infty\) norm of the density. Therefore, obtaining the upper bound for the density is the key to obtaining global solutions. We are interested in understanding the role that swirl plays in preventing density blow-up within our model. Is it positive? The answer is yes. Rewriting the transport equation as
\begin{align}\label{transport}
(\theta+\xi)_t+u_r\partial_r(\theta+\xi)+\int_R^r\frac{\rho}{s}(u_r^2-u_\theta^2)ds+P-\bar{P}+G(R,t)=0,
\end{align}
where 
\begin{align*}
\theta(\rho)=2\mu\log\rho+\frac{1}{\beta}\rho^\beta,\quad\xi=\int_R^r\rho u_rds,\quad G=(2\mu+\lambda)\div\mathbf{u}-(P-\bar{P}).
\end{align*}
The sign-preserving property of $u_\theta^2$ indicates that, in some sense, swirl and pressure are similar, as both can prevent finite-time density blow-up. In addition, we also need to estimate the other terms in $\eqref{transport}$. 

In order to derive the uniform upper bound of the density independent of time, we divide the proof into four cases.

\begin{enumerate}

\item 
When \(\beta>\max\left\{1,\frac{\gamma+2}{4}\right\}\), we adapt the method from \cite{2016 Huang-Li-JMPA} and \cite{2023 Huang-Su-Yan-Yu}. First, by taking advantage of the special structure of equation $\eqref{Equ 2}_2$ under radial symmetry, we provide the boundary representation of the effective viscous flux \(G\) (see Lemma \ref{F boundary}). Then, to estimate \(\norm{\xi}_{L^\infty L^\infty}\), we prove that the growth of \(\log (e+\norm{\nabla \mathbf{u}}_{L^2})\) is at most \(R_T^{1+\kappa+\varepsilon}\), where $R_T=\sup_{0\le t\le T}\norm{\rho(t)}_{L^\infty}+1$ and \(\kappa=\max\left\{0,\gamma-2\beta, \beta-\gamma-2\right\}\) (see Proposition \ref{suplog}). At this point, the lack of boundary conditions for the effective viscous flux \(G\) and the vorticity \(\omega\) becomes our biggest challenge. We note that $\nabla G=\partial_r G\mathbf{e_r}$ and $\nabla^\perp w=\partial_r w\mathbf{e_\theta}$, which allows us to directly take the inner product of the momentum equation with \(\mathbf{e_r}\) and \(\mathbf{e_\theta}\) respectively to decouple \(G\) and \(\omega\). Therefore, for any $1\leq p<\infty$, we get
\begin{align*}
&\norm{\nabla G}_{L^p}=\norm{\nabla G\cdot \frac{\mathbf{x}}{r}}_{L^p}=\norm{\rho\dot{\mathbf{u}}\cdot\frac{\mathbf{x}}{r}}_{L^p}\leq \norm{\rho\dot{\mathbf{u}}}_{L^p},\\
&\mu\norm{\nabla^\perp \omega}_{L^p}=\mu\norm{\nabla^\perp \omega\cdot\frac{\mathbf{x}^\perp}{r}}_{L^p}=\norm{\rho\dot{\mathbf{u}}\cdot\frac{\mathbf{x^\perp}}{r}}_{L^p}\leq \norm{\rho\dot{\mathbf{u}}}_{L^p}.
\end{align*}
Using the estimate of \(\log(e+\norm{\nabla \mathbf{u}}_{L^2})\), we prove that the growth of \(\norm{\xi}_{L^\infty L^\infty}\) will not exceed \(R_T^{1+\frac{\kappa}{2}+\varepsilon}\) through a logarithmic-type embedding inequality. Therefore, using standard methods, we have proven the uniform upper bound of the density (see Proposition \ref{uniform bound1}).

\item 
 When \(1< \beta\leq \gamma\), we first prove that $\left|\int_0^t G(R,\tau)d\tau\right|$ has a time-independent bound (see Proposition \ref{G boundary}). Then, inspired by \cite{2023 Fan-Wang-Li-arXiv}, by fully utilizing the damping mechanism of \(P\), we prove the time-independent higher integrability of \(\rho\) (see Proposition \ref{new potential}). This ensures that the growth of \(\log(e+\norm{\nabla \mathbf{u}}_{L^2})\) is at most \(R_T^{1+\varepsilon}\) (see Proposition \ref{new log}), which improves the estimate of $\xi$. Finally, using the Zlotnik inequality, we obtain the uniform upper bound of the density (see Proposition \ref{new upper bound}). Note that \(\gamma>\frac{\gamma+2}{4}\), so we obtain the asymptotic behavior of the solution when $\beta>1$.

 \item 
When $\beta=1$, the methods mentioned above become ineffective. Because it seems we can only achieve control of \(\log(e+\norm{\nabla \mathbf{u}}_{L^2})\) by superlinear growth of \(R_T\), rather than sublinear growth. This results in \(\norm{\xi}_{L^\infty L^\infty}\) being controlled only by the superlinear growth of \(R_T\). Therefore, We can only obtain
\begin{align*}
R_T\leq CR_T^{1+\varepsilon},
\end{align*}
which leads to a failure in obtaining the upper bound on the density. Therefore, to overcome this difficulty, we changed the way of estimating $\xi$. First, utilizing the radially symmetric structure and delicate index analysis, we prove that a uniform $L^p$ estimate for the density still holds in the endpoint case (see Proposition \ref{new potential}). This relaxes the restriction of $\beta > 1$ when establishing the uniform $L^p$ estimate for the density in \cite{2023 Fan-Wang-Li-arXiv}. Next, we used the idea from \cite{2023 Fan-Wang-Li-arXiv} to prove that \(\int_\Omega\rho|\mathbf{u}|^{2+\delta}dx\) has a time-independent bound, where $\delta=C(\mu,\gamma)R_T^{-\frac{1}{2}}$ and $C(\mu,\gamma)<\max\left\{\frac{1}{2},\frac{\gamma-1}{2(\gamma+1)}\right\}$ (see Proposition \ref{prop u}). Then, we estimate \(\norm{\xi}_{L^\infty L^\infty}\) as follows:
 \begin{align*}
\norm{\xi(t)}_{L^\infty}&\leq \int_0^{R}|\rho u_r|ds\leq \int_0^{R}\left(s\rho |u_r|^{2+\delta}\right)^{\frac{1}{2+\delta}}\rho^{\frac{1+\delta}{2+\delta}}s^{-\frac{1}{2+\delta}}ds\\
&\leq C\norm{\rho(t)}_{L^\infty}^{\frac{1+\delta}{2+\delta}}\left(\int_0^{R} s^{-\frac{1}{1+\delta}}ds\right)^{\frac{1+\delta}{2+\delta}}\leq CR_T^{\frac{1+\delta}{2+\delta}}\delta^{-\frac{1+\delta}{2+\delta}}\\
&\leq CR_T^{\frac{3}{2}\frac{1+\delta}{2+\delta}}\leq CR_T^{\frac{9}{10}}.
\end{align*}
With the estimate of \(\xi\) in hand, we obtained the uniform upper bound of the density using standard methods (see Proposition \ref{rho upper}).

\item 
When $\beta\in(0,1)$, we first assume \(\tilde{R}_T \leq (7\mu)^{\frac{1}{\beta}}\), where $\tilde{R}_T=\sup_{0\le t\le T}\norm{\rho(t)}_{L^\infty}$. Under this restriction, we use the fine structure of 2D radial symmetry to prove that \(\int_\Omega\rho |\mathbf{u}|^3dx\) has a time-independent bound that can be controlled by \(\norm{\rho_0}_{L^\infty}\) (see Proposition \ref{u high2}), which ensures that \(\norm{\xi}_{L^\infty L^\infty}\) has a time-independent bound. Then, we carefully estimate the boundary integral of the effective viscous flux \(G\), proving that it can be controlled by the initial mass \(M_0\) and energy \(E_0\). With these estimates at hand, we employ the Zlotnik inequality to obtain 
{\small
\begin{align*}
\begin{split}
2\mu\log \tilde{R}_T+\frac{1}{\beta}\tilde{R}_T^\beta&\leq \max\left\{2\mu\log\norm{\rho_0}_{L^\infty}+\frac{1}{\beta}\norm{\rho_0}_{L^\infty}^\beta,2\mu\log((\frac{(\gamma-1)E_0}{\pi R^2})^{\frac{1}{\gamma}})+\frac{1}{\beta}(\frac{(\gamma-1)E_0}{\pi R^2})^{\frac{\beta}{\gamma}}\right\}\\
&\quad+2^{\frac{4}{3}}R^{\frac{1}{3}}\pi^{-\frac{1}{3}}(7\mu)^{\frac{2}{3\beta}}\left(\int_\Omega\rho_0|\mathbf{u}_0|^3dx+3\sqrt{2}7^{\frac{\gamma}{\beta}}\mu^{\frac{\gamma}{\beta}-1} RE_0\right)^{\frac{1}{3}}\\
&\quad+\frac{1}{2 \mu\pi}(7\mu)^{\frac{1}{\beta}}E_0+\frac{1}{2\pi R}(M_0+2E_0)+\frac{2}{(\pi R^2)^\beta(1-\beta)}M_0^\beta+\frac{M_0 E_0}{4\mu\pi^2 R^2}.\\
\end{split}
\end{align*}
}
We observe that as \(\norm{\rho_0}_{L^\infty} \to 0^+\), the first term on the right-hand side of the inequality tends to \(-\infty\), while the other terms tend to \(0\). Therefore, when \(\norm{\rho_0}_{L^\infty}\leq a_0(\mu,\beta,\gamma,R,\norm{\nabla \mathbf{u_0}}_{L^2})\), we prove \(\tilde{R}_T \leq \frac{99}{100}(7\mu)^{\frac{1}{\beta}}\) (see Proposition \ref{Tilda R_T}), which ensures that the solution does not blow up in finite time.
\end{enumerate}

After obtaining the upper bound estimate for the density, we used the standard method from \cite{2016 Huang-Li-JMPA} to derive estimates for the higher-order derivatives of the velocity field and the density field (see Proposition \ref{First},\ref{Second},\ref{Third}). Then, we constructed the initial values $(\rho^\delta_0,\mathbf{u}_0^\delta)$ for the approximating system and proved that the approximating solutions $(\rho^\delta,\mathbf{u}^\delta)$ are global. At this point, due to the requirements for the initial values of local existence (see Lemma \ref{local existence}), we must close the estimate for \(\|\nabla^2 \rho^\delta\|_{L^2}\). Because of the structure of the transport equation, this necessitates estimating \(\|\nabla^3 \mathbf{u}^\delta\|_{L^2}\). To achieve this, following Hoff's method, we need to estimate \(\|\nabla^2 G^\delta\|_{L^2}\) and \(\|\nabla^2 w^\delta\|_{L^2}\). The absence of boundary conditions for the effective viscous flux and vorticity once again presents a challenge. We observe that, under the radially symmetric setting, \(\nabla^2 G^\delta\) and \(\Delta G^\delta\) have the following representations:
\begin{align*}
\partial_{ij} G^\delta&=\partial_j\left(\partial_r G^\delta\frac{x_i}{r}\right)=\partial_{rr}G^\delta\frac{x_ix_j}{r^2}+\partial_r G^\delta\frac{\delta_{ij}r^2-x_ix_j}{r^3},
\end{align*}
and 
\begin{align*}
\Delta G^\delta=\partial_{rr}G^\delta+\frac{\partial_rG^\delta}{r}.
\end{align*}
This means we can control \(\nabla^2 G^\delta\) using \(\Delta G^\delta\) and a tail term. Specifically,
\begin{align*}
\begin{split}
\norm{\nabla^2 G^\delta}_{L^2}&\leq \hat{C}\left(\norm{\partial_{rr} G^\delta}_{L^2}+\norm{\frac{\partial_r G^\delta}{r}}_{L^2}\right)\leq \hat{C}\left(\norm{\Delta G^\delta}_{L^2}+\norm{\frac{\partial_r G^\delta}{r}}_{L^2}\right)\\
&\leq \hat{C}\left(\norm{\nabla(\rho^\delta\dot{\mathbf{u}}^\delta)}_{L^2}+\norm{\frac{\rho^\delta \dot{\mathbf{u}}^\delta}{r}}_{L^2}\right)\leq \hat{C}\norm{\nabla \dot{\mathbf{u}}^\delta}_{L^2},
\end{split}
\end{align*}
where we have used $\Delta G^\delta=\div(\rho^\delta\dot{\mathbf{u}}^\delta)$ and the following simple fact
\begin{align*}
|\nabla\dot{\mathbf{u}}^\delta|^2=(\partial_r \left<\dot{\mathbf{u}}^\delta, \mathbf{e_r}\right>)^2+(\partial_r \left<\dot{\mathbf{u}}^\delta, \mathbf{e_\theta}\right>)^2+\frac{|\left<\dot{\mathbf{u}}^\delta, \mathbf{e_r}\right>|^2}{r^2}+\frac{|\left<\dot{\mathbf{u}}^\delta, \mathbf{e_\theta}\right>|^2}{r^2}\ge\left|\frac{\dot{\mathbf{u}}^\delta}{r}\right|^2.
\end{align*}
The estimate for $\norm{\nabla^2\omega^\delta}_{L^2}$ is similar. After obtaining the global approximating solutions, we obtain \((\rho, \mathbf{u})\) through standard compactness arguments.

We organize the paper as follows: In Section 2, we collect some basic facts and inequalities used in the proof. In Section 3, we derive the necessary a priori estimates. Specifically, in Section 3.1, we derive the uniform upper bound of the density when \(\beta > \max\left\{1,\frac{\gamma+2}{4}\right\}\). In Section 3.2, we derive the uniform upper bound of the density when \(1<\beta \leq \gamma\). In Section 3.3, we derive the uniform upper bound of the density when \(\beta=1\). In Section 3.4, we derive the uniform upper bound of the density under the assumptions \(0 < \beta < 1\) and small initial density. In Section 3.5, we establish the higher-order estimates. In Section 4, we complete the proof of the main theorems.

\section{Preliminary}
In this section, we state some lemmas that will be used in the proof. The well-known result on local existence \cite{1993 Salvi-Straškraba,1980 Solonnikov-JMS} when the initial density is strictly away from vacuum could be stated as follows:
\begin{lema}\label{local existence}
Assume that $(\rho_0,\mathbf{u_0})$ satisfy
\begin{align*}
\inf_{x\in \Omega}\rho_0(x)>0,\quad\rho_0\in H^2, \quad\mathbf{u}_0\in H^2\cap H_0^1,
\end{align*}
Then there exists a small time $T>0$ and a constant $C_0>0$ depending only on $\Omega,\gamma,\beta,\mu,$ $\norm{(\rho_0,\mathbf{u}_0)}_{H^2},\inf_{x\in \Omega}\rho_0(x)$ such that there exists a unique strong solution $(\rho, \mathbf{u})$ to the problem \eqref{Equ1}-\eqref{boundary condition1} in $\Omega\times(0,T)$ satisfying 
\begin{align*}
\inf_{(x,t)\in \Omega\times(0,T)}\rho(x,t)>C_0.
\end{align*}
\end{lema}
Due to the special structure of 2D radial symmetry, we have the following lemma to estimate $\norm{\mathbf{u}}_{L^\infty}$.
\begin{lema}\label{u inf}
	Assuming the Dirichlet boundary condition \eqref{boundary condition2} holds, then it comes that
	\begin{align}\label{u infty}
	\norm{\mathbf{u}(t)}_{L^\infty}\leq \frac{1}{\sqrt{2\pi}}\norm{\nabla \mathbf{u}(t)}_{L^2}.
	\end{align}
\end{lema}
\begin{proof}
Direct calculations show that
\begin{align*}
|u_r(r,t)|^2&=\int_0^r 2u_r\partial_r u_rds=\int_0^r2\frac{u_r}{s}\partial_r u_r sds\leq \int_0^r\left(\partial_r u_r+\frac{u_r}{s}\right)^2sds\leq \frac{1}{2\pi}\norm{\div \mathbf{u}(t)}_{L^2}^2,
\end{align*}
and 
\begin{align*}
|u_\theta(r,t)|^2&=\int_0^r 2u_\theta\partial_r u_\theta ds=\int_0^r2\frac{u_\theta}{s}\partial_r u_\theta sds\leq \int_0^r\left(\partial_r u_\theta+\frac{u_\theta}{s}\right)^2sds\leq \frac{1}{2\pi}\norm{w}_{L^2}^2.
\end{align*}
Therefore, we obtain
\begin{align*}
|\mathbf{u}(x,t)|^2=|u_r(r,t)|^2+|u_\theta(r,t)|^2\leq \frac{1}{2\pi}\norm{\div \mathbf{u}(t)}_{L^2}^2+\frac{1}{2\pi}\norm{w(t)}_{L^2}^2=\frac{1}{2\pi}\norm{\nabla \mathbf{u}(t)}_{L^2}^2.
\end{align*}
\end{proof}

The following Gagliardo-Nirenberg inequality will be used frequently.
\begin{lema}\cite{1975 Adams}\label{G-N ineq}
For any $2<q<\infty$ and $f\in H^1(\Omega)$, there exists a constant  $C>0$ such that
\begin{align}\label{GN1}
\norm{f}_{L^q}\leq C\norm{f}_{L^2}^{\frac{2}{q}}\norm{ f}_{H^1}^{1-\frac{2}{q}}, \quad \norm{f}_{L^q}\leq C\norm{f}_{L^2}^{\frac{2}{q}}\norm{\nabla f}_{L^2}^{1-\frac{2}{q}}+C|\bar{f}|.
\end{align}
If $f\in W^{1,q}(\Omega)$, then
\begin{align}\label{GN2}
\norm{f}_{L^\infty}\leq C\norm{f}_{L^2}^{\frac{q-2}{2q-2}}\norm{f}_{W^{1,q}}^{\frac{q}{2q-2}}.
\end{align}
In particular, if $f|_{\partial \Omega}=0$ or $\int_\Omega fdx=0$, then $\norm{f}_{H^1}$ and $\norm{f}_{W^{1,q}}$ in the above equation can be replaced by $\norm{\nabla f}_{L^2}$ and $\norm{\nabla f}_{L^q}$ respectively.
\end{lema}

Furthermore, the following Beale-Kato-Majda type inequality, which can be found in \cite{1984 Beale-Kato-Majda-CMP}, plays an important role in obtaining the estimate of $\norm{\nabla \rho}_{L^q}$.
\begin{lema}\label{B-K-M}
For any $2<q<\infty$, there is a constant $C>0$ such that for all $\nabla\mathbf{u}\in W^{1,q}(\Omega)$, we have
\begin{align}
\label{BKM}
\norm{\nabla\mathbf{u}}_{L^\infty}\leq C(\norm{\div \mathbf{u}}_{L^\infty}+\norm{\curl \mathbf{u}}_{L^\infty})\log(e+\norm{\nabla^2\mathbf{u}}_{L^q})+C\norm{\nabla\mathbf{u}}_{L^2}+C.
\end{align}
\end{lema}

The following properties of the Bogovskii operator can be found in \cite{2023 Fan-Wang-Li-arXiv}.
\begin{lema}\label{Bogovskii}
Define
\[
L_0^p(\Omega) = \left\{ f | \norm{f}_{L^p} < \infty, \int_{\Omega} f \, dx = 0 \right\}.
\]
Then, for \(1 < p < \infty\), there is a bounded linear operator \( \mathcal{B} \) given by
\[
\mathcal{B} : L_0^p \rightarrow W_0^{1,p}, \quad f \mapsto \mathcal{B}(f),
\]
such that \( u = \mathcal{B}(f) \) is a solution to the equation below,
\begin{equation}\label{Bo Equ}
\begin{cases}
\div \, u = f & \text{in } \Omega, \\
u = 0 & \text{on } \partial \Omega.
\end{cases}
\end{equation}

Moreover, we have following properties:\\
(i)
 For $1<p<\infty$, there is a constant $C(p)$ depending on $\Omega$ and $p$, such that
\begin{align*}
\norm{\mathcal{B}(f)}_{W^{1,p}}\leq C(p)\norm{f}_{L^p}.
\end{align*}
(ii)
In particular, when $f=\div g$ with $g\cdot n=0$ on $\partial\Omega$, $\mathcal{B}(f)$ is well defined and satisfies that 
\begin{align*}
\norm{\mathcal{B}(f)}_{L^p}\leq C(p)\norm{g}_{L^p},
\end{align*}
and $u=\mathcal{B}(f)$is a weak solution to \eqref{Bo Equ}.

\end{lema}

Finally, we need the following Zlotnik inequality to obtain the uniform estimate of the density.

\begin{lema}\cite{2000 Zlotnik-JDE}\label{Zlotnik}
Let the function $y\in W^{1,1}(0,T)$ satisfy
\begin{align*}
y'(t)=g(y)+h'(t) \text{ on } [0,T], \quad y(0)=y^0
\end{align*}
with $g\in C(\mathbb{R})$ and $h\in W^{1,1}(0,T)$. If $g(\infty)=-\infty$ and 
\begin{align*}
h(t_2)-h(t_1)\leq N_0+N_1(t_2-t_1)
\end{align*}
for all $0\leq t_1<t_2\leq T$ with some $N_0\ge 0$ and $N_1\ge 0$, then
\begin{align*}
y(t)\leq \max\left\{y^0, \tilde{\zeta}\right\}+N_0<\infty \text{ on }[0,T],
\end{align*}
where $\tilde{\zeta}$ is a constant such that 
\begin{align*}
g(\zeta)\leq -N_1, \text{ for }\zeta\ge \tilde{\zeta}.
\end{align*}
\end{lema}

\section{A priori estimates}
Let $T>0$ be a fixed time and $(\rho,\mathbf{u})$ be a radially symmetric strong solution of the problem \eqref{Equ 2}-\eqref{boundary condition2} on $\Omega\times(0,T) $. We begin by performing standard energy equality.

\begin{lema}
Let
\begin{align*}
E(t)=\frac{1}{2}\int_{\Omega}\rho(t)|\mathbf{u}(t)|^2dx+\frac{1}{\gamma-1}\int_\Omega(\rho(t))^\gamma dx,
\end{align*}
then for any $t\in[0,T)$,
   \begin{align}\label{energy estimate}
       \begin{split}
            E(t)+\int_0^t\mu\norm{\nabla \mathbf{u}}_{L^2}^2+\norm{\sqrt{\mu+\lambda}\div\mathbf{u}}_{L^2}^2 d\tau=E(0).
       \end{split}
    \end{align}  
\end{lema}
Next, we present the renormalized transport equation structure first introduced by Kazhikhov and Vaigant \cite{1995 Vaigant-Kazhikhov-SMJ}. Define 
$\theta(\rho)=2\mu\log\rho +\frac{1}{\beta}\rho^\beta$
and the effective viscous flux as
\begin{align}\label{effective viscous flux}
G=(2\mu+\lambda)\div \mathbf{u}-(P-\bar{P}).
\end{align}
Multiplying the transport equation $\eqref{Equ1}_1$ with $\theta'(\rho)$, we have
\begin{align}\label{RTE}
\theta_t+u_r\partial_r\theta+G+P-\bar{P}=0.
\end{align}
Due to the definition of the effective viscous flux \eqref{effective viscous flux},  we can rewrite the momentum equation $\eqref{Equ 2}_2$ as
\begin{align}\label{RME}
(\rho u_r)_t+\partial_r(\rho u_r^2)+\frac{\rho}{r}(u^2_r-u_\theta^2)=\partial_rG.
\end{align}
Integrating the above equation from $R$ to $r$ yields
\begin{align}\label{RME 2}
\xi_t+\rho u_r^2+\int_R^r\frac{\rho}{s}(u_r^2-u_\theta^2)ds=G-G(R,t),
\end{align}
where $\xi=\int_R^r \rho u_rds$. Adding equation \eqref{RTE} and equation \eqref{RME 2} gives
\begin{align}\label{Flow  structure}
(\theta+\xi)_t+u_r\partial_r(\theta+\xi)+\int_R^r\frac{\rho}{s}(u_r^2-u_\theta^2)ds+P-\bar{P}+G(R,t)=0.
\end{align}

Following the idea in \cite{2023 Huang-Su-Yan-Yu}, we provide the boundary representation of the effective viscous flux.
\begin{lema}\label{F boundary}
It holds that
\begin{align}\label{FBT}
G(R,t)=\frac{1}{R^2}\left\{\frac{d}{dt}\int_0^R\rho u_rr^2dr+\int_0^R2G(r,t)rdr-\int_0^R\rho (u_r^2+u_\theta^2)rdr\right\}.
\end{align}
\end{lema}

\begin{proof}
Multiplying the momentum equation $\eqref{Equ 2}_2$ by $r^2$ gives 
\begin{align}
\partial_t(r^2\rho u_r)+\partial_r(r^2\rho u_r^2)-r(\rho u_r^2+\rho u_\theta^2)=r^2\partial_rG.
\end{align}
Integrating the above equation from $0$ to $R$ and integrating by parts yield
\begin{align*}
\frac{d}{dt}\int_0^R \rho u_rr^2dr-\int_0^R\rho  (u_r^2+u_\theta^2)rdr=R^2G(R,t)-\int_0^R 2G(r,t)rdr.
\end{align*}
Then
\begin{align*}
G(R,t)=\frac{1}{R^2}\left\{\frac{d}{dt}\int_0^R\rho u_rr^2dr+\int_0^R2G(r,t)rdr-\int_0^R\rho (u_r^2+u_\theta^2)rdr\right\}.
\end{align*}
\end{proof}

\subsection{Uniform upper bound of density for $\beta>\max\left\{1,\frac{\gamma+2}{4}\right\}$}
In this subsection, our aim is to establish a time-independent uniform upper bound of the density under assumption $\beta>\max\left\{1,\frac{\gamma+2}{4}\right\}$. Our proof follows the ideas in \cite{2016 Huang-Li-JMPA} and \cite{2023 Huang-Su-Yan-Yu}. We denote
\begin{align*}
	A_1^2(t)&=\int_\Omega\left(\frac{G^2}{2\mu+\lambda}+\mu\omega^2\right)dx,\\
	A_2^2(t)&=\int_\Omega\rho|\dot{\mathbf{u}}|^2dx,\\
	A_3^2(t)&=\int_\Omega\left((2\mu+\lambda)(\div \mathbf{u})^2+\mu\omega^2\right)dx,
\end{align*}
and 
\begin{align*}
R_T:=\sup_{0\le t\le T}\norm{\rho(t)}_{L^\infty}+1.
\end{align*}
We need the following lemma to obtain $L^p$ estimates for $\nabla G$ and $\nabla w$.
\begin{lema}
For any $1\le p<\infty$, there holds
\begin{align}\label{G,w}
\begin{split}
&\norm{\nabla G}_{L^p}\leq \norm{\rho\dot{\mathbf{u}}}_{L^p},\\
&\mu\norm{\nabla^\perp w}_{L^p}\leq \norm{\rho\dot{\mathbf{u}}}_{L^p}.
\end{split}
\end{align}
\end{lema}
\begin{proof}
Note that $G$, $w$ satisfy
\begin{align}\label{eff estimate}
\rho \dot{\mathbf{u}}=\nabla G+\mu\nabla^\perp\omega,
\end{align}
and the following simple fact
\begin{align*}
&\nabla G=\partial_r G \mathbf{e_r},\\
&\nabla^\perp w=\partial_r w\mathbf{e_\theta}.
\end{align*}
We take the inner product of equation \eqref{eff estimate} with $\mathbf{e_r}$ and $\mathbf{e_\theta}$ respectively to get
\begin{align*}
&\norm{\nabla G}_{L^p}=\norm{\nabla G\cdot \frac{\mathbf{x}}{r}}_{L^p}=\norm{\rho\dot{\mathbf{u}}\cdot\frac{\mathbf{x}}{r}}_{L^p}\leq \norm{\rho\dot{\mathbf{u}}}_{L^p},\\
&\mu\norm{\nabla^\perp \omega}_{L^p}=\mu\norm{\nabla^\perp \omega\cdot\frac{\mathbf{x}^\perp}{r}}_{L^p}=\norm{\rho\dot{\mathbf{u}}\cdot\frac{\mathbf{x^\perp}}{r}}_{L^p}\leq \norm{\rho\dot{\mathbf{u}}}_{L^p}.
\end{align*}
This completes the proof.
\end{proof}
\begin{prop}\label{suplog}
For any $\alpha\in(0,1)$, there exists a constant $C>0$ depending only on $\alpha,\mu,\beta,\gamma,R,\norm{\rho_0}_{L^\infty},\norm{\mathbf{u_0}}_{H^1}$ such that 
\begin{align}\label{the first structure}
\sup_{0\le t\le T}\log(e+A_1^2(t)+A_3^2(t))+\int_0^T\frac{A_2^2(t)}{e+A_1^2(t)}dt\leq CR_T^{1+\kappa+\alpha\beta},
\end{align}
where $\kappa=\max\left\{0,\gamma-2\beta, \beta-\gamma-2\right\}$.
\end{prop}
\begin{proof}
Rewrite the momentum equation $\eqref{Equ1}_2$ as
\begin{align}\label{RME2}
\rho \dot{\mathbf{u}}=\nabla G+\mu\nabla^\perp\omega.
\end{align}
Note that 
\begin{align*}
\div\dot{\mathbf{u}}&=\frac{D}{Dt}\div\mathbf{u}+(\div\mathbf{u})^2-2\nabla u_1\cdot \nabla^\perp u_2\\
&=\frac{D}{Dt}\left(\frac{G+P-\bar{P}}{2\mu+\lambda}\right)+(\div\mathbf{u})^2-2\partial_r u_r\frac{u_r}{r}-2\partial_r u_\theta\frac{u_\theta}{r},
\end{align*}
and 
\begin{align*}
\nabla^\perp\cdot\dot{\mathbf{u}}=\frac{D}{Dt}w+w\div\mathbf{u}.
\end{align*}
Multiplying equation \eqref{RME2} by $\dot{\mathbf{u}}$ and integrating by parts gives
\begin{align}\label{I_0}
\begin{split}
\frac{1}{2}\frac{d}{dt}A_1^2+A_2^2&=\frac{1}{2}\int_\Omega \frac{2\mu+\lambda-\beta\lambda}{(2\mu+\lambda)^2}G^2\div\mathbf{u}dx-\int_\Omega \frac{\beta \lambda(P-\bar{P})}{(2\mu+\lambda)^2}G\div \mathbf{u}dx+\int_\Omega\frac{\gamma P}{2\mu+\lambda}G\div \mathbf{u}dx\\
&\quad -(\gamma-1)\int_\Omega\frac{G}{2\mu+\lambda}dx\int_\Omega P\div \mathbf{u}dx-\int_\Omega G(\div \mathbf{u})^2dx\\
&\quad+2\int_\Omega G\nabla u_1\cdot\nabla^\perp u_2dx-\frac{\mu}{2}\int_\Omega w^2\div \mathbf{u}dx\\
&=\sum_{i=1}^{7}I_i.
\end{split}
\end{align}
Next, we estimate $I_1$-$I_7$. By using H\"{o}lder inequality and \eqref{GN1}, we get
\begin{align}\label{I_1}
\begin{split}
|I_1|&\leq C\norm{\div\mathbf{u}}_{L^2}\norm{\frac{G^2}{2\mu+\lambda}}_{L^2}\leq C\norm{\div\mathbf{u}}_{L^2}\norm{\frac{G}{\sqrt{2\mu+\lambda}}}_{L^4}^2\\
&\leq CA_3\norm{\frac{G}{\sqrt{2\mu+\lambda}}}_{L^2}^{1-\alpha}\norm{G}_{L^{2(1+\alpha)/\alpha}}^{1+\alpha}\\
&\leq C A_1A_3R_T^{\frac{\alpha\beta}{2}}\norm{G}_{H^1}.
\end{split}
\end{align}
H\"{o}lder inequality and \eqref{energy estimate} imply that
\begin{align}\label{I_2+I_3}
\begin{split}
|I_2|+|I_3|&\leq C\int \frac{P+1}{2\mu+\lambda}|G||\div \mathbf{u}|dx\\
&\leq C\norm{\frac{P+1}{(2\mu+\lambda)^{\frac{3}{2}}}}_{L^{2+\frac{\beta}{\gamma}}}\norm{G}_{L^{2+\frac{4\gamma}{\beta}}}\norm{\sqrt{2\mu+\lambda}\div\mathbf{u}}_{L^2}\\
&\leq CR_T^{\max\left\{0,\frac{\gamma}{2}-\beta\right\}}\norm{G}_{H^1}A_3.
\end{split}
\end{align}
By using the energy equality \eqref{energy estimate}, one has
\begin{align}\label{I_4}
\begin{split}
|I_4|&=(\gamma-1)\left|\int_\Omega \frac{P-\bar{P}}{2\mu+\lambda}dx\int_\Omega(P-\bar{P})\div\mathbf{u}dx\right|\\
&=(\gamma-1)\left|\int_\Omega \frac{P-\bar{P}}{2\mu+\lambda}dx\int_\Omega(G-(2\mu+\lambda)\div\mathbf{u})\div\mathbf{u}dx\right|\\
&\leq C\norm{G}_{L^2}A_3+CA_3^2.
\end{split}
\end{align}
Next, it follows from $\eqref{I_1}$ and $\eqref{I_2+I_3}$ that
\begin{align}\label{I_5}
\begin{split}
|I_5|&\leq \int_\Omega \frac{G^2}{2\mu+\lambda}|\div\mathbf{u}|dx+\int_\Omega\frac{|P-\bar{P}|}{2\mu+\lambda}|G||\div\mathbf{u}|dx\\
&\leq C A_1A_3R_T^{\frac{\alpha\beta}{2}}\norm{G}_{H^1}+CR_T^{\max\left\{0,\frac{\gamma}{2}-\beta\right\}}\norm{G}_{H^1}A_3.
\end{split}
\end{align}
Next, the radial symmetry structure and \eqref{boundary condition2} indicates that
\begin{align}\label{I_6}
\begin{split}
|I_6|&=2\left|\int_\Omega G\left(\partial_r u_r\frac{u_r}{r}+\partial_r u_\theta\frac{u_\theta}{r}\right)dx\right|\\
&=4\pi\left|\int_0^R G\left(\partial_r u_ru_r+\partial_r u_\theta u_\theta\right)dr\right|\\
&=2\pi \left|\int_0^R \partial_rG(u_r^2+u_\theta^2)dr\right|\\
&\leq C\norm{\nabla G}_{L^2}\norm{\mathbf{u}}_{L^\infty}\left(\norm{\frac{u_r}{r}}_{L^2}+\norm{\frac{u_\theta}{r}}_{L^2}\right)\\
&\leq C\norm{\nabla G}_{L^2}\norm{\nabla \mathbf{u}}_{L^2}^2\\
&\leq C\norm{\nabla G}_{L^2}A_1A_3+CR_T^{\max\left\{\frac{\gamma}{2}-\beta,0\right\}}\norm{\nabla G}_{L^2}A_3,
\end{split}
\end{align}
where we used
\begin{align*}
\norm{\nabla\mathbf{u}}_{L^2}\leq C\norm{\frac{G}{2\mu+\lambda}}_{L^2}+C\norm{\frac{P-\bar{P}}{2\mu+\lambda}}_{L^2}+C\norm{\omega}_{L^2}\leq CA_1+CR_T^{\max\left\{\frac{\gamma}{2}-\beta,0\right\}}.
\end{align*}
Finally, Gagliardo-Nirenberg inequality \eqref{GN1} implies that
\begin{align}\label{I_7}
|I_7|\leq C\norm{\div \mathbf{u}}_{L^2}\norm{w}_{L^4}^2\leq  C\norm{w}_{H^1}A_1A_3,
\end{align}
Substituting \eqref{I_1}-\eqref{I_7} into \eqref{I_0}, we obtain
\begin{align}\label{new I_0}
\begin{split}
&\quad\frac{d}{dt}A_1^2+2A_2^2\\
&\leq C (R_T^{\frac{\alpha\beta}{2}}\norm{G}_{H^1}+\norm{w}_{H^1})A_1A_3+CR_T^{\max\left\{0,\frac{\gamma}{2}-\beta\right\}}\norm{G}_{H^1}A_3+CA_3^2\\
&\leq \varepsilon R_T^{-1}(\norm{G}_{H^1}^2+\norm{w}_{H^1}^2)+CR_T^{1+\alpha\beta}A_1^2A_3^2+CR_T^{\max\left\{1,\gamma-2\beta+1\right\}}A_3^2\\
&\leq A_2^2+CR_T^{1+\alpha\beta}A_1^2A_3^2+CR_T^{\max\left\{1,\gamma-2\beta+1,\beta-\gamma-1\right\}}A_3^2,
\end{split}
\end{align}
where we used
\begin{align*}
|\bar{G}|=\left|\int_\Omega\lambda\div\mathbf{u}dx\right|\leq \norm{\sqrt{\lambda}}_{L^2}A_3\leq R_T^{\max\left\{\frac{\beta-\gamma}{2},0\right\}}A_3.
\end{align*}
Note that the following simple fact shows the relationship between $A_1$ and $A_3$,
\begin{align}\label{A_1A_3}
\begin{split}
A_3^2&\leq C\norm{\sqrt{2\mu+\lambda}\div\mathbf{u}}_{L^2}^2+C\norm{w}_{L^2}^2\\
&\leq C\norm{\frac{G}{\sqrt{2\mu+\lambda}}}_{L^2}^2+C\norm{\frac{P-\bar{P}}{\sqrt{2\mu+\lambda}}}_{L^2}^2+C\norm{w}_{L^2}^2\\
&\leq CA_1^2+CR_T^{\max\left\{\gamma-\beta,0\right\}}.
\end{split}
\end{align}
Dividing equation \eqref{new I_0} by $e+A_1^2$ shows
\begin{align*}
\frac{d}{dt}\log(e+A_1^2)+\frac{A_2^2}{e+A_1^2}\leq CR_T^{\max\left\{1+\alpha\beta,1,\gamma-2\beta+1,\beta-\gamma-1\right\}}A_3^2.
\end{align*}
Integrating with respect to time gives
\begin{align}\label{A_1}
\sup_{0\le t\le T}\log(e+A_1^2)+\int_0^T\frac{A_2^2}{e+A_1^2}dt\leq CR_T^{1+\kappa+\alpha\beta}.
\end{align}
Therefore, the above equation combined with $\eqref{A_1A_3}$ completes the proof.
\end{proof}

Next, we will employ the Zlotnik inequality to obtain the uniform upper bound of the density.
\begin{prop}\label{uniform bound1}
Let $\beta>\max\left\{1,\frac{\gamma+2}{4}\right\}$, there exists a constant $C$ depending only on $\mu,\beta,\gamma,R$ and initial data, but not on $T$, such that for any $T>0$, 
\begin{align}\label{uniform bound 1}
\sup_{0\le t\le T}\norm{\rho (t)}_{L^\infty}\leq C.
\end{align}
\end{prop}
\begin{proof}
Reviewing the previous transport structure
\begin{align*}
\frac{D}{Dt}(\theta+\xi)+\int_R^r\frac{\rho}{s}(u_r^2-u_\theta^2)ds+P-\bar{P}+G(R,t)=0.
\end{align*}
Let 
\begin{align*}
y(t)&=\theta(t,x(t)),\quad g(y)=-P(\theta^{-1}(y)),\\
h(t)&=-\xi(t,x(t))+\int_0^t\int_r^R\frac{\rho }{s}(u_r^2-u_\theta^2)dsd\tau+\int_0^t\bar{P}(\tau)d\tau-\int_0^tG(R,\tau)d\tau.
\end{align*}
Simple calculations show that
\begin{align}\label{h estimate}
\begin{split}
|h(t_2)-h(t_1)|&\leq 2\norm{\xi}_{L^\infty L^\infty}+\int_{t_1}^{t_2}\int_r^R\frac{\rho }{s}(u_r^2+u_\theta^2)dsd\tau+\int_{t_1}^{t_2}\bar{P}d\tau+\left|\int_{t_1}^{t_2}G(R,\tau)d\tau\right|\\
&:=\sum_{i=1}^{4}J_i.
\end{split}
\end{align}
Next, we estimate $J_1$-$J_4$. By Proposition \ref{suplog}, for some $q>2$, we have 
\begin{align*}
\begin{split}
\norm{\xi(t)}_{L^\infty}&\leq C\norm{\rho \mathbf{u}}_{L^2}\log^{\frac{1}{2}}(e+\norm{\rho \mathbf{u}}_{L^q})+C\norm{\rho\mathbf{u}}_{L^{\frac{2\gamma}{\gamma+1}}}+C\\
&\leq CR_T^{\frac{1}{2}}\log^{\frac{1}{2}}(e+R_T+\norm{\nabla\mathbf{u}}_{L^2})+C\\
&\leq CR_T^{1+\frac{\kappa}{2}+\frac{\alpha\beta}{2}},
\end{split}
\end{align*}
which means that
\begin{align*}
J_1\leq CR_T^{1+\frac{\kappa}{2}+\frac{\alpha\beta}{2}}.
\end{align*}
By using the energy equality \eqref{energy estimate}, we obtain
\begin{align*}
J_2\leq CR_T\int_0^T\int_\Omega|\nabla \mathbf{u}|^2dxdt\leq CR_T,
\end{align*}
where we used the following simple fact
\begin{align*}
|\nabla \mathbf{u}|^2=(\partial_r u_r)^2+(\partial_r u_\theta)^2+\frac{u_r^2}{r^2}+\frac{u_\theta^2}{r^2}.
\end{align*}
The energy equality $\eqref{energy estimate}$ indicates that
\begin{align*}
J_3\leq C(t_2-t_1).
\end{align*}
Finally, we estimate the boundary term $J_4$. Recalling that
\begin{align*}
G(R,t)=\frac{1}{R^2}\left\{\frac{d}{dt}\int_0^R\rho u_rr^2dr+\int_0^R2rG(r,t)dr-\int_0^R\rho(u_r^2+u_\theta^2)rdr\right\}
\end{align*}
Employing the energy equality \eqref{energy estimate} again, we obtain
\begin{align*}
J_4&\leq C\left|\int_0^R\rho u_r r^2dr(t_2)-\int_0^R\rho u_r r^2dr(t_1)\right|+C\left|\int_{t_1}^{t_2}\int_\Omega \rho^\beta \div\mathbf{u}dxd\tau\right|\\
&\quad+C\left|\int_{t_1}^{t_2}\int_0^R\rho (u_r^2+u_\theta^2)rdrdt\right|\\
&\leq C+CR_T^{\beta-\gamma}+C(t_2-t_1),
\end{align*}
where we used 
\begin{align*}
\partial_t(\rho^\beta)+\div(\rho^\beta\mathbf{u})+(\beta-1)\rho^\beta\div\mathbf{u}=0.
\end{align*}
Substituting the above estimates into \eqref{h estimate} gives
\begin{align*}
|h(t_2)-h(t_1)|\leq CR_T^{1+\frac{\kappa}{2}+\frac{\alpha\beta}{2}}+CR_T^{\beta-\gamma}+C(t_2-t_1).
\end{align*}
By using Lemma \ref{Zlotnik}, we obtain
\begin{align}\label{Zlotnik bound}
\theta(t,x(t))\leq CR_T^{1+\frac{\kappa}{2}+\frac{\alpha\beta}{2}}+CR_T^{\beta-\gamma},
\end{align}
taking the supremum in time and space completes the proof.
\end{proof}
\subsection{Uniform upper bound of density for $1<\beta\leq \gamma$}
In this subsection, we prove the uniform upper bound of the density when $1<\beta\leq \gamma$. First, the following lemma shows that the boundary integral of $G$ can be independent of time.
\begin{prop}\label{G boundary}
Suppose that $\gamma>1$, $1\leq \beta\leq \gamma$, there exists a constant $C$ depending only on $\mu,\beta,\gamma,R$ and initial data, but not on $T$, such for any $t\in [0,T]$,
\begin{align}\label{H est}
\left|\int_0^t G(R,\tau)d\tau\right|\leq C.
\end{align}
\end{prop}
\begin{proof}
Integrating \eqref{FBT} over \((0, t)\) gives
\begin{align}
\begin{split}
\int_0^t G(R,\tau)d\tau&=\dfrac{1}{R^2}\left[\int_0^R\rho u_rr^2dr|_{\tau=t}-\int_0^R\rho u_rr^2dr|_{\tau=0}+2\int_0^t\int_0^RG(r,\tau)rdrd\tau\right.\\
&\quad\left. -\int_0^t\int_0^R\rho (u_r^2+u_\theta^2)rdrd\tau\right]\\
&:=K_1+K_2+K_3+K_4.
\end{split}
\end{align}
By using the energy equality \eqref{energy estimate}, we obtain the estimates for the terms \(K_1\) and \(K_2\),
\begin{align*}
|K_1|\leq C\int_0^R(\rho(t)+\rho u_r^2(t))rdr\leq C,
\end{align*}
and 
\begin{align*}
|K_2|\leq C\int_0^R(\rho_0+\rho_0 u_{r_0}^2)rdr\leq C.
\end{align*}
Note that when $1<\beta\leq \gamma$,
\begin{align*}
\int_0^t\int_0^RG(r,\tau)rdrd\tau&=\frac{1}{2\pi}\int_0^t\int_\Omega\rho^\beta\div\mathbf{u}dxd\tau=\frac{1}{2\pi(1-\beta)}\int_0^t\frac{d}{dt}\int_\Omega\rho^\beta dxd\tau\\
&=\frac{1}{2\pi(1-\beta)}\left(\int_\Omega \rho^\beta(t)dx-\int_\Omega \rho_0^\beta dx\right).
\end{align*}
and when $\beta=1$,
\begin{align*}
\int_0^t\int_0^RG(r,\tau)rdrd\tau&=\frac{1}{2\pi}\int_0^t\int_\Omega\rho\div\mathbf{u}dxd\tau=-\frac{1}{2\pi}\int_0^t\frac{d}{dt}\int_\Omega\rho\log\rho dxd\tau\\
&=-\frac{1}{2\pi}\left(\int_\Omega \rho\log\rho(t)dx-\int_\Omega \rho_0\log\rho_0dx\right).
\end{align*}
Therefore, we use equality \eqref{energy estimate} to obtain the estimate for \(K_3\),
\begin{align*}
|K_3|\leq C.
\end{align*}
Using the energy equality \eqref{energy estimate} and \eqref{u infty} , one has the estimate for \(K_4\),
\begin{align*}
|K_4|\leq C\int_0^t\norm{\mathbf{u}}_{L^\infty}^2\int_0^R\rho rdrd\tau\leq C\int_0^t\norm{\nabla \mathbf{u}}_{L^2}^2d\tau\leq C. 
\end{align*}
Combining the estimates for \(K_1\) to \(K_4\), we complete the proof.
\end{proof}

Next, we use a key idea from \cite{2023 Fan-Wang-Li-arXiv} to obtain the time-independent higher integrability of the density field.
\begin{prop}\label{new potential}
Suppose that $\gamma>1$, $1\leq \beta\leq \gamma$, for any $k\geq 1$, there exists positive constants $M$ and $C$ depending only on $\mu,\beta,\gamma,k ,R$ and initial data, but not on $T$, such that
\begin{align}\label{rho High}
\sup_{0\leq t\leq T}\int_\Omega \rho^{k\beta\gamma+1}dx\leq C.
\end{align}
\end{prop}
\begin{proof}
Define $H(t)=\int_0^t G(R,s)ds$, there is also a structure similar to equation \eqref{Flow  structure}:
\begin{align}\label{FS2}
(\theta+\xi+H(t))_t+u_r\partial_r(\theta+\xi+H(t))+\int_R^r\frac{\rho }{s}(u_r^2-u_\theta^2)ds+P(r,t)-\bar{P}(t)=0.
\end{align}
Let $f=(\theta+\xi+H(t)-M)_+$, where $M$ is a sufficiently large constant to be determined. Multiplying \eqref{FS2} by $\rho f^{k\gamma-1}$, one obtains 
\begin{align}\label{Rho f 2gamma}
\begin{split}
\dfrac{1}{k\gamma}\dfrac{d}{dt}\int_\Omega \rho f^{k\gamma}dx&=\int_\Omega\rho f^{k\gamma-1}\dfrac{D}{Dt}fdx\\
&=\int_\Omega \rho f^{k\gamma-1}\left(\int_r^R\frac{\rho}{s}(u_r^2-u_\theta^2)ds-P(r,t)+\bar{P}\right)dx.
\end{split}
\end{align}
Next, we claim that under the condition $f>0$, there exists constants $C$ such that
\begin{align}\label{damping}
f^{\frac{\gamma}{\beta}}+\left(\frac{M}{2}\right)^\frac{\gamma}{\beta}\leq C\left(f+\frac{M}{2}\right)^{\frac{\gamma}{\beta}}=C\left(\theta+\xi+H(t)-\frac{M}{2}\right)^{\frac{\gamma}{\beta}}\leq C(P+|\xi|^{\frac{\gamma}{\beta}}),
\end{align}
where we choose $M$ sufficiently large such that $|H(t)|\leq \frac{M}{2}$ , using \eqref{H est}. In fact, observing $\theta+\xi$ is positive, we obtain
\begin{align*}
\left(\theta+\xi\right)^{\frac{\gamma}{\beta}}=\left(2\mu\log\rho+\frac{1}{\beta}\rho^\beta+\xi\right)^{\frac{\gamma}{\beta}}\leq C(\rho^\beta+|\xi|)^{\frac{\gamma}{\beta}}\leq C(P+|\xi|^{\frac{\gamma}{\beta}}).
\end{align*}
Substituting inequality $\eqref{damping}$ into $\eqref{Rho f 2gamma}$ yields
\begin{align}\label{Rho f 2gamma 2}
\begin{split}
&\quad\dfrac{1}{k\gamma}\dfrac{d}{dt}\int_\Omega \rho f^{k\gamma}dx+\int_\Omega \rho f^{k\gamma-1}\left(f^{\frac{\gamma}{\beta}}+\left(\frac{M}{2}\right)^\frac{\gamma}{\beta}\right)dx\\
&\leq C \int_\Omega\rho f^{k\gamma-1}\left(\int_r^R\frac{\rho}{s}u_r^2ds\right)dx+C\int_\Omega\rho f^{k\gamma-1}\bar{P}dx+C\int_\Omega\rho f^{k\gamma-1}|\xi|^{\frac{\gamma}{\beta}}dx\\
&=N_1+N_2+N_3.
\end{split}
\end{align}
By using H\"older inequality and Sobolev inequality, we get
\begin{align}\label{N1}
\begin{split}
N_1&\leq C\norm{\rho^{\frac{1}{k\gamma}}f}_{L^{k\gamma}}^{k\gamma-1}\norm{\rho}_{L^{k\beta\gamma+1}}^{\frac{1}{k\gamma}}\norm{\int_r^R\frac{\rho}{s}u_r^2ds}_{L^{\frac{k\beta\gamma+1}{\beta}}}\\
&\leq C\norm{\rho^{\frac{1}{k\gamma}}f}_{L^{k\gamma}}^{k\gamma-1}\norm{\rho}_{L^{k\beta\gamma+1}}^{\frac{1}{k\gamma}}\norm{\frac{\rho u_r^2}{r}}_{L^{\frac{2(k\beta\gamma+1)}{k\beta\gamma+1+2\beta}}}\\
&\leq C\norm{\rho^{\frac{1}{k\gamma}}f}_{L^{k\gamma}}^{k\gamma-1}\norm{\rho}_{L^{k\beta\gamma+1}}^{\frac{1}{k\gamma}}\norm{\rho u_r}_{L^{\frac{k\beta \gamma+1}{\beta}}}\norm{\frac{u_r}{r}}_{L^2}.
\end{split}
\end{align}
Note that 
\begin{align*}
\norm{\rho u_r}_{L^{\frac{k\beta\gamma+1}{\beta}}}&\leq \norm{\rho }_{L^{\frac{k\beta\gamma+1}{\beta}}} \norm{u_r}_{L^\infty}\leq C \norm{\rho}_{L^{k\beta\gamma+1}}\norm{\div \mathbf{u}}_{L^2},
\end{align*}
and
\begin{align*}
\norm{\frac{u_r}{r}}_{L^2}^2=2\pi\int_0^R\frac{u_r^2}{r^2}rdr\leq2\pi \int_0^R\left(\partial_r u_r+\frac{u_r}{r}\right)^2rdr=\int_\Omega|\div\mathbf{u}|^2dx.
\end{align*}
Therefore, we obtain the estimate for \(N_1\):
\begin{align*}
N_1\leq C\left(1+\int_\Omega\rho f^{k\gamma}dx+\int_\Omega \rho^{k\beta\gamma+1}dx\right)\norm{\div \mathbf{u}}_{L^2}^2.
\end{align*}
By using the equality $\eqref{energy estimate}$, we obtain
\begin{align}\label{N2}
N_2\leq C\int_\Omega \rho f^{k\gamma-1}\bar{P}dx\leq C\int_\Omega\rho f^{k\gamma-1}dx.
\end{align}
Finally, we estiamate \(N_3\) in two cases: $\frac{\gamma}{\beta}\leq 2$ and $\frac{\gamma}{\beta}>2$.

When $\frac{\gamma}{\beta}\leq 2$, if $\gamma>2$, then by the Sobolev embedding, \eqref{u infty} and \eqref{energy estimate} yields
\begin{align*}
\norm{\xi}_{L^\infty}^2&=\norm{\int_R^r\rho u_r ds}_{L^\infty}^2\leq C \norm{\rho u_r}_{L^{\gamma}}^2\leq C \norm{\rho}_{L^\gamma}^2\norm{\div\mathbf{u}}_{L^2}^2\leq C\norm{\div\mathbf{u}}_{L^2}^2.
\end{align*}
Then we estimate \(N_3\) as follows:
\begin{align}\label{N3 0}
\begin{split}
N_3&\leq C\int_\Omega\rho f^{k\gamma-1}(1+|\xi|^2)dx\\
&\leq C(1+\norm{\xi}_{L^\infty}^2)\int_\Omega \rho f^{k\gamma-1}dx\\
&\leq C(1+\norm{\div\mathbf{u}}_{L^2}^2)\int_\Omega \rho f^{k\gamma-1}dx.
\end{split}
\end{align}
If $\gamma\leq 2$, we choose \(\varepsilon\) sufficiently small such that 
\begin{align*}
\norm{\xi}_{L^\infty}^2&=\norm{\int_R^r\rho u_r ds}_{L^\infty}^2\leq C\norm{\rho u_r}_{L^{2+\varepsilon}}^2\leq C\norm{\rho}_{L^\gamma}^{2-2\theta}\norm{\rho}_{L^{k\beta\gamma+1}}^{2\theta}\norm{\div\mathbf{u}}_{L^2}^2,
\end{align*}
where
\begin{align*}
\theta=\frac{k\beta\gamma+1}{k\beta\gamma+1-\gamma}\left(1-\frac{\gamma}{2+\varepsilon}\right).
\end{align*}
Therefore, we estimate \(N_3\) as follows:
\begin{align}\label{N3 1}
\begin{split}
N_3&\leq C\int_\Omega\rho f^{k\gamma-1}(1+|\xi|^2)dx\\
&\leq C(1+\norm{\xi}_{L^\infty}^2)\int_\Omega \rho f^{k\gamma-1}dx\\
&\leq C(1+\norm{\rho}_{L^{k\beta\gamma+1}}^{2\theta}\norm{\div\mathbf{u}}_{L^2}^2)\int_\Omega \rho f^{k\gamma-1}dx\\
&\leq C\norm{\rho}_{L^{k\beta\gamma+1}}^{2\theta}\norm{\div\mathbf{u}}_{L^2}^2\left(\int_\Omega \rho f^{k\gamma}dx\right)^{\frac{k\gamma-1}{k\gamma}}+C\int_\Omega \rho f^{k\gamma-1}dx\\
&\leq C\left(1+\int_\Omega\rho f^{k\gamma}dx+\int_\Omega \rho^{k\beta\gamma+1}dx\right)\norm{\div\mathbf{u}}_{L^2}^2+C\int_\Omega \rho f^{k\gamma-1}dx,
\end{split}
\end{align}
where we used 
\begin{align*}
\beta\ge 1,\quad \gamma>1,
\end{align*}
and 
\begin{align*}
\frac{k\gamma-1}{k\gamma}+\frac{2}{k\beta\gamma+1-\gamma}\left(1-\frac{\gamma}{2+\varepsilon}\right)\leq 1,
\end{align*}
which is equivalent to
\begin{align*}
\beta\ge-\frac{2\gamma}{2+\varepsilon}+\left(2+\frac{1}{k}\right)-\frac{1}{k\gamma}.
\end{align*}

When $\frac{\gamma}{\beta}>2$, we choose \(\varepsilon\) sufficiently small such that
\begin{align*}
\norm{\xi}_{L^\infty}&=\norm{\int_R^r\rho u_r ds}_{L^\infty}\leq C\norm{\rho u_r}_{L^{2+\varepsilon}}\leq C\norm{\sqrt{\rho}u_r}_{L^2}^{1-\eta}\norm{\sqrt{\rho}}_{L^{2\gamma}}^{1+\eta}\norm{u_r}_{L^\infty}^\eta,
\end{align*}
where 
\begin{align*}
\eta=\frac{1}{\gamma-1}\left(\gamma+1-\frac{2\gamma}{2+\varepsilon}\right)\in(0,1).
\end{align*}
At this point, we estimate \(N_3\) as follows:
\begin{align}\label{N3 2}
\begin{split}
N_3&\leq C\int_\Omega\rho f^{k\gamma-1}|\xi|^{\frac{\gamma}{\beta}}dx\leq C\norm{\xi}_{L^\infty}^{\frac{\gamma}{\beta}}\int_\Omega \rho f^{k\gamma-1}dx\leq C\norm{\div\mathbf{u}}_{L^2}^{\frac{\eta\gamma}{\beta}}\int_\Omega \rho f^{k\gamma-1}dx.
\end{split}
\end{align}
Let $\frac{\eta\gamma}{\beta}=2$, then 
\begin{align*}
\frac{2\beta}{\gamma}=\frac{1}{\gamma-1}\left(\gamma+1-\frac{2\gamma}{2+\varepsilon}\right),
\end{align*}
we point out that such an \(\varepsilon\) exists, due to our assumptions 
\begin{align*}
\beta \ge 1,\quad \gamma>2.
\end{align*}

Substituting the estimates for \(N_1\) to \(N_3\) back into $\eqref{Rho f 2gamma 2}$ gives
\begin{align*}
\begin{split}
&\quad\frac{1}{k\gamma}\dfrac{d}{dt}\int_\Omega \rho f^{k\gamma}dx+\int_\Omega \rho f^{k\gamma-1}\left(f^{\frac{\gamma}{\beta}}+\left(\frac{M}{2}\right)^\frac{\gamma}{\beta}\right)dx\\
&\leq C\left(1+\int_\Omega\rho f^{k\gamma}dx+\int_\Omega \rho^{k\beta\gamma+1}dx\right)\norm{\div\mathbf{u}}_{L^2}^2+C\int_\Omega \rho f^{k\gamma-1}dx.
\end{split}
\end{align*}
Let $M$ be sufficiently large such that 
\begin{align*}
C\int_\Omega \rho f^{k\gamma-1}dx\leq \frac{1}{2}\int_\Omega\rho f^{k\gamma-1}\left(\frac{M}{2}\right)^{\frac{\gamma}{\beta}}dx,
\end{align*}
therefore, we get
\begin{align}\label{Rho f2gamma}
\begin{split}
&\quad\frac{1}{k\gamma}\dfrac{d}{dt}\int_\Omega \rho f^{k\gamma}dx+\frac{1}{2}\int_\Omega \rho f^{k\gamma-1}\left(f^{\frac{\gamma}{\beta}}+\left(\frac{M}{2}\right)^\frac{\gamma}{\beta}\right)dx\\
&\leq C\left(1+\int_\Omega\rho f^{k\gamma}dx+\int_\Omega \rho^{k\beta\gamma+1}dx\right)\norm{\div\mathbf{u}}_{L^2}^2.
\end{split}
\end{align}

Finally, by using $\eqref{H est}$, H\"{o}lder inequality and Sobolev embedding, we estimate $\int_\Omega\rho^{k\beta\gamma+1}dx$ as follows:
\begin{align*}
\begin{split}
\int_\Omega\rho^{k\beta\gamma+1}dx&=\int_{\Omega\cap(\rho<2)}\rho^{k\beta\gamma+1}dx+\int_{\Omega\cap(\rho>2)}\rho^{k\beta\gamma+1}dx\\
&\leq C+C\int_{\Omega}\rho f^{k\gamma}dx+C\int_{\Omega}\rho|\xi|^{k\gamma}dx\\
&\leq C+C\int_{\Omega}\rho f^{k\gamma}dx+C\norm{\rho}_{L^{\frac{2(k\beta\gamma+1)}{2k\beta\gamma+2-k\gamma}}}\norm{\xi}_{L^{2(k\beta\gamma+1)}}^{k\gamma}\\
&\leq C+C\int_{\Omega}\rho f^{k\gamma}dx+C\norm{\rho}_{L^{\frac{2(k\beta\gamma+1)}{2k\beta\gamma+2-k\gamma}}}\norm{\rho u_r}_{L^{\frac{2(k\beta\gamma+1)}{k\beta\gamma+2}}}^{k\gamma}\\
&\leq C+C\int_{\Omega}\rho f^{k\gamma}dx+C\norm{\rho}_{L^{k\beta\gamma+1}}\norm{\rho^{\frac{1}{2}}}_{L^{2(k\beta\gamma+1)}}^{k\gamma}\norm{\rho^{\frac{1}{2}}u_r}_{L^2}^{k\gamma}\\
&\leq C+C\int_{\Omega}\rho f^{k\gamma}dx+C\norm{\rho}_{L^{k\beta\gamma+1}}^{\frac{k\gamma}{2}+1}.
\end{split}
\end{align*}
The Young's inequality shows that
\begin{align}\label{Rho 2bg+1}
\int_\Omega\rho^{k\beta \gamma+1}dx\leq C+C\int_{\Omega}\rho f^{k\gamma}dx.
\end{align}
Substituting \eqref{Rho 2bg+1} into \eqref{Rho f2gamma}, we obtain
\begin{align}\label{Dissipation}
\begin{split}
&\frac{1}{k\gamma}\dfrac{d}{dt}\int_\Omega \rho f^{k\gamma}dx+\frac{1}{2}\int_\Omega \rho f^{k\gamma-1}\left(f^{\frac{\gamma}{\beta}}+\left(\frac{M}{2}\right)^\frac{\gamma}{\beta}\right)dx\leq C\left(1+\int_\Omega\rho f^{k\gamma}dx\right)\norm{\div\mathbf{u}}_{L^2}^2.
\end{split}
\end{align}
We complete the proof by using the Gronwall inequality and \eqref{Rho 2bg+1}.
\end{proof}
The Proposition \ref{new potential} essentially establishes a new potential under the condition $1\leq \beta\leq \gamma$. Therefore, naturally, we can obtain the following Proposition similar to Proposition \ref{suplog}.
\begin{prop}\label{new log}
Suppose that $1\leq \beta\leq\gamma$. For any $\alpha\in(0,1)$, there exists a constant $C>0$ depending only on $\alpha,\mu,\beta,\gamma,R,\norm{\rho_0}_{L^\infty},\norm{\mathbf{u_0}}_{H^1}$ such that 
\begin{align}
\sup_{0\le t\le T}\log(e+A_1^2(t)+A_3^2(t))+\int_0^T\frac{A_2^2(t)}{e+A_1^2(t)}dt\leq CR_T^{1+\alpha\beta}.
\end{align}
\end{prop}
\begin{proof}
The proof is similar to that of Proposition \ref{suplog}, but for completeness, we still provide some key steps of the proof.
Multiplying equation \eqref{RME2} by $\dot{\mathbf{u}}$ and integrating by parts gives
\begin{align}\label{2I_0}
\begin{split}
\frac{1}{2}\frac{d}{dt}A_1^2+A_2^2&=\frac{1}{2}\int_\Omega \frac{2\mu+\lambda-\beta\lambda}{(2\mu+\lambda)^2}G^2\div\mathbf{u}dx-\int_\Omega \frac{\beta \lambda(P-\bar{P})}{(2\mu+\lambda)^2}G\div \mathbf{u}dx+\int_\Omega\frac{\gamma P}{2\mu+\lambda}G\div \mathbf{u}dx\\
&\quad -(\gamma-1)\int_\Omega\frac{G}{2\mu+\lambda}dx\int_\Omega P\div \mathbf{u}dx-\int_\Omega G(\div \mathbf{u})^2dx\\
&\quad+2\int_\Omega G\nabla u_1\cdot\nabla^\perp u_2dx-\frac{\mu}{2}\int_\Omega w^2\div \mathbf{u}dx\\
&=\sum_{i=1}^{7}I_i.
\end{split}
\end{align}
The estimates for \(I_1\), \(I_4\) and \(I_7\)  are the same as in Proposition \ref{suplog}.
For \(I_2\) and \(I_3\), we estimate as follows:
\begin{align}\label{2I_2+I_3}
\begin{split}
|I_2|+|I_3|&\leq C\int_\Omega \frac{P+1}{2\mu+\lambda}|G||\div \mathbf{u}|dx\\
&\leq C\norm{\frac{P+1}{(2\mu+\lambda)^{\frac{3}{2}}}}_{L^{2+\varepsilon}}\norm{G}_{L^{\frac{2(2+\varepsilon)}{\varepsilon}}}\norm{\sqrt{2\mu+\lambda}\div\mathbf{u}}_{L^2}\\
&\leq C\norm{G}_{H^1}A_3,
\end{split}
\end{align}
where we used \eqref{rho High}. Similarly,
\begin{align}\label{2I_5}
\begin{split}
|I_5|&\leq \int_\Omega \frac{G^2}{2\mu+\lambda}|\div\mathbf{u}|dx+\int_\Omega\frac{|P-\bar{P}|}{2\mu+\lambda}|G||\div\mathbf{u}|dx\\
&\leq C A_1A_3R_T^{\frac{\alpha\beta}{2}}\norm{G}_{H^1}+C\norm{G}_{H^1}A_3,
\end{split}
\end{align}
and 
\begin{align}\label{2I_6}
|I_6|\leq C\norm{\nabla G}_{L^2}\norm{\nabla \mathbf{u}}_{L^2}^2\leq C\norm{\nabla G}_{L^2}A_3^2.
\end{align}
Substituting \eqref{I_1},\eqref{2I_2+I_3},\eqref{I_4},\eqref{2I_5},\eqref{2I_6} and \eqref{I_7} into \eqref{2I_0} gives
\begin{align}\label{2new I_0}
\begin{split}
&\quad\frac{d}{dt}A_1^2+2A_2^2\\
&\leq C (R_T^{\frac{\alpha\beta}{2}}\norm{G}_{H^1}+\norm{w}_{H^1})A_1A_3+C\norm{G}_{H^1}A_3+CA_3^2+C\norm{\nabla G}_{L^2}A_3^2\\
&\leq \varepsilon R_T^{-1}(\norm{G}_{H^1}^2+\norm{w}_{H^1}^2)+CR_T^{1+\alpha\beta}A_1^2A_3^2+CR_TA_3^2+CR_TA_3^4\\
&\leq A_2^2+CR_T^{1+\alpha\beta}A_1^2A_3^2+CR_TA_3^2+CR_TA_3^4,
\end{split}
\end{align}
where we used
\begin{align*}
|\bar{G}|=\left|\int_\Omega\lambda\div\mathbf{u}dx\right|\leq \norm{\sqrt{\lambda}}_{L^2}A_3\leq CA_3.
\end{align*}
Note the following relationship between \(A_1\) and \(A_3\),
\begin{align}\label{2A_1A_3}
\begin{split}
A_3^2&\leq C\norm{\sqrt{2\mu+\lambda}\div\mathbf{u}}_{L^2}^2+C\norm{w}_{L^2}^2\\
&\leq C\norm{\frac{G}{\sqrt{2\mu+\lambda}}}_{L^2}^2+C\norm{\frac{P-\bar{P}}{\sqrt{2\mu+\lambda}}}_{L^2}^2+C\norm{w}_{L^2}^2\\
&\leq CA_1^2+C.
\end{split}
\end{align}
Dividing equation \eqref{2new I_0} by $e+A_1^2$ implies that
\begin{align*}
\frac{d}{dt}\log(e+A_1^2)+\frac{A_2^2}{e+A_1^2}\leq CR_T^{1+\alpha\beta}A_3^2.
\end{align*}
Integrating with respect to time gives
\begin{align}\label{2A_1}
\sup_{0\le t\le T}\log(e+A_1^2)+\int_0^T\frac{A_2^2}{e+A_1^2}dt\leq CR_T^{1+\alpha\beta}.
\end{align}
Using \eqref{2A_1A_3} again, we complete the proof.
\end{proof}
Having obtained the above Proposition, we can immediately prove the following Proposition using the method in Proposition \ref{uniform bound1}, simply by replacing \(\kappa\) with \(0\) in inequality \eqref{Zlotnik bound}. Therefore, we omit the proof.
\begin{prop}\label{new upper bound}
Let $1<\beta\leq \gamma$, there exists a constant $C$ depending only on $\mu,\beta,\gamma,R$ and initial data, but not on $T$, such that for any $T>0$, 
\begin{align}\label{uniform bound2}
\sup_{0\le t\le T}\norm{\rho (t)}_{L^\infty}\leq C.
\end{align}
\end{prop}

\subsection{Uniform upper bound of density for $\beta=1$}
In this subsection, we will obtain a uniform upper bound for the density under the condition \( \beta = 1 \), which is entirely different from the case where \( \beta > 1 \). 

We first prove the time-independent spatial higher integrability of the density field, which is a natural corollary of Proposition \ref{new potential}.

\begin{prop}
Suppose that $\beta=1, k\gamma\ge 3$, then there exists a sufficiently large constant \(M_1\) and \(C\), depending on $\mu,\beta,\gamma,R,k$ and initial data, but not on time $T$, such that
\begin{align}\label{uniform spatial}
\int_0^T\int_\Omega (\rho-M_1)_+^{k\gamma-1}dxdt\leq C.
\end{align}
\end{prop}
\begin{proof}
Assume that \( M \) is the constant in \eqref{Dissipation}. Let $M_1\ge\frac{3}{2}M$, then we obtain from \( f=(\theta+\xi+H-M)_+ \) that 
\begin{align*}
\rho\leq f-2\mu\log\rho-\xi-H+M\leq f+|\xi|+\frac{3M}{2} \leq f+|\xi|+M_1,
\end{align*}
when $\rho\ge M_1$. Therefore, 
\begin{align*}
(\rho-M_1)_+\leq f+|\xi|.
\end{align*}
Direct calculations show that
\begin{align}\label{xi}
\begin{split}
\int_0^T\int_\Omega \rho(\rho-M_1)_+^{k\gamma-1}dxdt\leq C\int_0^T\int_\Omega \rho f^{k\gamma-1}dxdt+C\int_0^T\int_\Omega\rho|\xi|^{k\gamma-1}dxdt,
\end{split}
\end{align}
and
\begin{align}
\int_0^T\int_\Omega \rho(\rho-M_1)_+^{k\gamma-1}dxdt\ge\int_0^T\int_\Omega (\rho-M_1)_+^{k\gamma-1}dxdt.
\end{align}
From \eqref{Dissipation}, we know that the first term on the right-hand side of \eqref{xi} is bounded, so we estimate the second term as follows:
\begin{align}
\begin{split}
C\int_0^T\int_\Omega\rho|\xi|^{k\gamma-1}dxdt&\leq C\norm{\rho}_{L^\infty L^\gamma}\norm{\xi}_{L^{k\gamma-1}L^{\frac{\gamma(k\gamma-1)}{\gamma-1}}}^{k\gamma-1}\leq C\norm{\rho u_r}_{L^{k\gamma-1}L^{\frac{2\gamma(k\gamma-1)}{\gamma(k\gamma-1)+2(\gamma-1)}}}^{k\gamma-1}\\
&\leq C\norm{u_r}_{L^2L^\infty}^2\norm{\sqrt{\rho}u_r}_{L^\infty L^2}^{k\gamma-3}\norm{\sqrt{\rho}}_{L^\infty L^{\frac{\gamma(k\gamma+1)}{2\gamma-1}}}^{k\gamma+1}\leq C,
\end{split}
\end{align}
where we used
\begin{align*}
\frac{\gamma(k\gamma+1)}{2\gamma-1}\leq 2k\gamma.
\end{align*}
Thus, we have completed the proof.
\end{proof}

We utilize the Lemma \ref{Bogovskii} to prove the following Proposition.
\begin{prop}\label{Prop Bogovskii}
Suppose that $\beta=1$, there exists a constant $C$ depending on $\mu,\beta,\gamma,R$ and initial data, but not on time $T$, such that
\begin{align}\label{rho-rho_s}
\int_0^T\int_\Omega (\rho^{\gamma-1}+1)(\rho-\rho_s)^2dxdt\leq C.
\end{align}
\end{prop}

\begin{proof}
Taking the inner product of $\eqref{Equ1}_2$ with \(\mathcal{B}(\rho-\rho_s)\) and integrating by parts yields
\begin{align}\label{rho gamma- rhos gamma}
\begin{split}
\int_\Omega (\rho^\gamma-\rho_s^\gamma)(\rho-\rho_s)dx&=\frac{d}{dt}\int_\Omega \rho \mathbf{u}\cdot\mathcal{B}(\rho-\rho_s)dx-\int_\Omega \rho\mathbf{u}\cdot\mathcal{B}(\partial_t\rho)dx\\
&\quad-\int_\Omega \rho\mathbf{u}\otimes\mathbf{u}:\nabla \mathcal{B}(\rho-\rho_s)dx+\mu\int_\Omega \nabla\mathbf{u}:\nabla \mathcal{B}(\rho-\rho_s)dx\\
&\quad+\int_\Omega(\mu+\lambda)\div\mathbf{u}(\rho-\rho_s)dx=\sum_{i=0}^{4}V_i.
\end{split}
\end{align}
The H\"{o}lder inequality, \eqref{rho High} and Lemma \ref{Bogovskii} indicate that
\begin{align}\label{V_1}
V_1\leq C\norm{\rho \mathbf{u}}_{L^2}^2\leq C\norm{\nabla \mathbf{u}}_{L^2}^2,
\end{align}
and 
\begin{align}\label{V_2}
V_2\leq C\int_\Omega\rho|\mathbf{u}|^2|\nabla\mathcal{B}(\rho-\rho_s)|dx\leq C\norm{\rho}_{L^2}\norm{\nabla\mathcal{B}(\rho-\rho_s)}_{L^2}\norm{\nabla \mathbf{u}}_{L^2}^2\leq C\norm{\nabla \mathbf{u}}_{L^2}^2.
\end{align}
By Young's inequality, we obtain
\begin{align}\label{V_3}
V_3\leq C\norm{\nabla\mathcal{B}(\rho-\rho_s)}_{L^2}\norm{\nabla \mathbf{u}}_{L^2}\leq \varepsilon\norm{\rho-\rho_s}_{L^2}^2+C\norm{\nabla\mathbf{u}}_{L^2}^2,
\end{align}
and
\begin{align}\label{V_4}
\begin{split}
V_4&\leq C\int_\Omega(\rho-M_1)_+|\div\mathbf{u}||\rho-\rho_s|+M_1 |\div\mathbf{u}||\rho-\rho_s|dx\\
&\leq \varepsilon\int_\Omega|\rho-\rho_s|^{\gamma+1}+|\rho-\rho_s|^2 dx+C\int_\Omega (\rho-M_1)_+^{\frac{2(\gamma+1)}{\gamma-1}}+|\div\mathbf{u}|^2dx.
\end{split}
\end{align}
Note the following simple fact:
\begin{align*}
\begin{split}
\int_\Omega (\rho^\gamma-\rho_s^\gamma)(\rho-\rho_s)dx\ge C\int_\Omega(\rho^{\gamma-1}+1)(\rho-\rho_s)^2dx.
\end{split}
\end{align*}
Therefore, substituting \eqref{V_1}-\eqref{V_4} into equation \eqref{rho gamma- rhos gamma}, we get
\begin{align}\label{new rho gamma- rhos gamma}
\begin{split}
C\int_\Omega(\rho^{\gamma-1}+1)(\rho-\rho_s)^2dx&\leq \frac{d}{dt}\int_\Omega \rho \mathbf{u}\cdot\mathcal{B}(\rho-\rho_s)dx+C\int_\Omega (\rho-M_1)_+^{\frac{2(\gamma+1)}{\gamma-1}}dx+C\norm{\nabla \mathbf{u}}_{L^2}^2.
\end{split}
\end{align}
Finally, we use \eqref{rho High} to obtain
\begin{align*}
\left|\int_\Omega \rho\mathbf{u}\cdot\mathcal{B}(\rho-\rho_s)dx\right|\leq \norm{\sqrt{\rho}}_{L^4}\norm{\sqrt{\rho}\mathbf{u}}_{L^2}\norm{\mathcal{B}(\rho-\rho_s)}_{L^4}\leq C.
\end{align*}
Thus, integrating \eqref{new rho gamma- rhos gamma} with respect to time, we complete the proof with \eqref{uniform spatial}.
\end{proof}

The following Proposition, similar to Proposition 3.4 in \cite{2023 Fan-Wang-Li-arXiv}, shows the higher integrability of the velocity field, which plays an important role in bounding the upper bound of the density.
\begin{prop}\label{prop u}
Suppose that $\beta=1$, there exists constants $C(\mu,\gamma)\in\left(0,\max\left\{\frac{1}{2},\frac{\gamma-1}{2(\gamma+1)}\right\}\right)$ depend only on $\mu,\gamma$ and $C$ depend only on $\mu,\beta,\gamma,R$ and initial data, but not on $T$, such that for all $\delta\leq C(\mu,\gamma)R_T^{-\frac{1}{2}}$, 
	\begin{equation}  \label{u high}
	\sup_{0\leq t\leq T}\int_{\Omega} \rho |\mathbf{u}|^{2+\delta}dx\leq C.
	\end{equation}
\end{prop}

\begin{proof}
Multiplying $\eqref{Equ1}_2$ by $(2+\delta)|\mathbf{u}|^\delta \mathbf{u}$, and integrating over $\Omega$ we get
\begin{align*}
\begin{split}
& \frac{d}{dt} \int_\Omega \rho|\mathbf{u}|^{2+\delta}dx+\mu(2+\delta) \int_\Omega|\mathbf{u}|^\delta(|\nabla \mathbf{u}|^2+\delta|
\nabla|\mathbf{u}||^2)dx+(2+\delta)\int_\Omega(\mu+\lambda)|\mathbf{u}|^\delta|\div \mathbf{u}|^2dx \\
& \leq C \int_\Omega |\rho^\gamma-\rho_s^\gamma||\mathbf{u}|^\delta|\nabla \mathbf{u}|dx+\delta(2+\delta) \int_\Omega(\mu+\lambda)|\div \mathbf{u} ||\mathbf{u}|^\delta|\nabla|\mathbf{u}||dx \\
& \leq C \int_\Omega |\rho^\gamma-\rho_s^\gamma||\mathbf{u}|^\delta|\nabla \mathbf{u}|dx+ \frac{\mu(2+\delta)}{2} \int_\Omega|\mathbf{u}|^\delta|\nabla| \mathbf{u}||^2dx+\frac{1}{2\mu}\delta^2(2+\delta)  \int_\Omega(\mu+\lambda)^2|\mathbf{u}|^\delta|\operatorname{div} \mathbf{u}|^2dx.
\end{split}
\end{align*}
Note that 
\begin{align*}
(2+\delta)(\mu+\lambda)>\frac{1}{2\mu}\delta^2(2+\delta) (\mu+\lambda)^2,
\end{align*}
due to 
\begin{align*}
\delta^2\leq (C(\mu,\gamma))^2R_T^{-1}< \frac{2\mu}{\mu+R_T}\leq \frac{2\mu}{\mu+\lambda}.
\end{align*}
Therefore,
\begin{align}\label{rho u 2+delta}
\begin{split}
& \frac{d}{dt} \int_\Omega\rho|\mathbf{u}|^{2+\delta}dx+\dfrac{\mu(2+\delta)}{2} \int_\Omega|\mathbf{u}|^\delta|\nabla \mathbf{u}|^2dx+\mu\delta(2+\delta)\int_\Omega|\mathbf{u}|^\delta|
\nabla|\mathbf{u}||^2dx \\
& \leq C \int_\Omega (\rho^{\gamma-1}+1)|\rho-\rho_s||\mathbf{u}|^\delta|\nabla \mathbf{u}|dx\\
&\leq C\int_\Omega (\rho-M_1)_+^{\gamma-1}|\rho-\rho_s||\mathbf{u}|^\delta|\nabla \mathbf{u}|dx+C\int_\Omega M_1^{\gamma-1}|\rho-\rho_s||\mathbf{u}|^\delta|\nabla \mathbf{u}|dx\\
&\leq C\int_\Omega|\rho-\rho_s|^{\gamma+1}+(\rho-M_1)_+^q+|\nabla\mathbf{u}|^2+|\mathbf{u}|^2dx+C\int_\Omega |\rho-\rho_s|^{p}+|\mathbf{u}|^2+|\nabla\mathbf{u}|^2dx,
\end{split}
\end{align}
where $q$, $p$ satisfy
\begin{align*}
\frac{\gamma-1}{q}+\frac{1}{\gamma+1}+\frac{\delta}{2}+\frac{1}{2}=1,\quad \frac{1}{p}+\frac{\delta}{2}+\frac{1}{2}=1.
\end{align*}
Due to $0<\delta<\frac{\gamma-1}{2(\gamma+1)}$, one has
\begin{align*}
q\in\left(2\gamma+2,4\gamma+4\right),\quad p\in\left(2,\frac{4\gamma+4}{\gamma+3}\right).
\end{align*}
We calculate
\begin{align*}
|\rho-\rho_s|^{\gamma+1}\leq C(\rho^{\gamma-1}+1)(\rho-\rho_s)^2,
\end{align*}
and 
\begin{align*}
|\rho-\rho_s|^p\leq  C(\rho^{\gamma-1}+1)(\rho-\rho_s)^2.
\end{align*}
Integrating \eqref{rho u 2+delta} with respect to time, we complete the proof using \eqref{uniform spatial} and \eqref{rho-rho_s}.
\end{proof}

With the above lemmas in hand, we can now derive the uniform upper bound of the density.
\begin{prop}\label{rho upper}
If $\beta= 1$, there exists a constant $C$ depend on $\mu,\beta,\gamma,R$ and initial data, but not on time $T$, such that
\begin{align}\label{upper bound density}
\sup_{0\le t\le T}\norm{\rho(t)}_{L^\infty}\leq C.
\end{align}
\end{prop}
\begin{proof}
The proof is similar to that of Proposition \ref{uniform bound1}, we only need to estimate each term in \eqref{h estimate}, for any $0\leq t_1<t_2\leq t\leq T$,
\begin{align}\label{h estimate2}
\begin{split}
|h(t_2)-h(t_1)|&\leq 2\norm{\xi}_{L^\infty L^\infty}+\int_{t_1}^{t_2}\int_r^R\frac{\rho }{s}(u_r^2+u_\theta^2)dsd\tau+\int_{t_1}^{t_2}\bar{P}d\tau+\left|\int_{t_1}^{t_2}G(R,\tau)d\tau\right|\\
&:=\sum_{i=1}^{4}J_i.
\end{split}
\end{align}
Let \(\delta = C(\mu,\gamma)R_T^{-\frac{1}{2}}\), where \(C(\mu,\gamma)\) is the constant from Proposition \ref{prop u}. The H\"older inequality shows that
\begin{align*}
\norm{\xi(\tau)}_{L^\infty}&\leq \int_0^{R}|\rho u_r|ds\leq \int_0^{R}\left(s\rho |u_r|^{2+\delta}\right)^{\frac{1}{2+\delta}}\rho^{\frac{1+\delta}{2+\delta}}s^{-\frac{1}{2+\delta}}ds\\
&\leq C\norm{\rho(\tau)}_{L^\infty}^{\frac{1+\delta}{2+\delta}}\left(\int_0^{R} s^{-\frac{1}{1+\delta}}ds\right)^{\frac{1+\delta}{2+\delta}}\leq CR_T^{\frac{1+\delta}{2+\delta}}\delta^{-\frac{1+\delta}{2+\delta}}\\
&\leq CR_T^{\frac{3}{2}\frac{1+\delta}{2+\delta}}\leq CR_T^{\frac{9}{10}},
\end{align*}
where we used $C(\mu,\gamma)\leq \frac{1}{2}$. Then we get
\begin{align}\label{J_1}
J_1\leq CR_T^{\frac{9}{10}}.
\end{align}
Note that 
\begin{align*}
|\nabla \mathbf{u}|^2=(\partial_r u_r)^2+(\partial_r u_\theta)^2+\frac{u_r^2}{r^2}+\frac{u_\theta^2}{r^2}.
\end{align*}
Then, from the energy equality \eqref{energy estimate}, we obtain
\begin{align}\label{J_2}
J_2\leq C\int_{0}^{t}\norm{\rho(\tau)}_{L^\infty}\norm{\nabla \mathbf{u}}_{L^2}^2d\tau,
\end{align}
and 
\begin{align}\label{J_3}
J_3\leq C(t_2-t_1).
\end{align}
Proposition \ref{G boundary} shows 
\begin{align}\label{J_4}
J_4\leq C.
\end{align}
Therefore, one has
\begin{align}
|h(t_2)-h(t_1)|\leq C+CR_T^{\frac{9}{10}}+C\int_{0}^{t}\norm{\rho(\tau)}_{L^\infty}\norm{\nabla \mathbf{u}}_{L^2}^2d\tau+C(t_2-t_1).
\end{align}
By employing the Zlotnik inequality, we obtain
\begin{align*}
\theta(t,x(t))\leq C+CR_T^{\frac{9}{10}}+C\int_{0}^{t}\norm{\rho(\tau)}_{L^\infty}\norm{\nabla \mathbf{u}}_{L^2}^2d\tau.
\end{align*}
Taking the supremum over space yields
\begin{align*}
\norm{\rho(t)}_{L^\infty}\leq C+CR_T^{\frac{9}{10}}+C\int_{0}^{t}\norm{\rho(\tau)}_{L^\infty}\norm{\nabla \mathbf{u}}_{L^2}^2d\tau.
\end{align*}
By using the Gronwall inequality, we derive
\begin{align*}
\norm{\rho(t)}_{L^\infty}\leq C+CR_T^{\frac{9}{10}}.
\end{align*}
Taking the supremum over time gives
\begin{align*}
R_T\leq C+CR_T^{\frac{9}{10}},
\end{align*}
which, combined with Young's inequality, completes the proof.
\end{proof}

\subsection{Uniform upper bound of small density for $0<\beta<1$}
In this subsection, we aim to prove that under condition $0<\beta<1$, when the \(L^\infty\) norm of the initial density is sufficiently small, the local strong solution \(\rho\) has a time-independent uniform upper bound. Moreover, we provide a precise bound of how small the initial density needs to be. We define
\begin{align*}
\tilde{R}_T=\sup_{0\le t\le T}\norm{\rho(t)}_{L^\infty}.
\end{align*}
The following Proposition shows that when \(\tilde{R}_T\leq C(\mu,\beta)\), the higher integrability of the velocity field has a time-independent bound and it can be controlled by \(\rho_0\).
\begin{prop}\label{u high2}
Suppose that 
\begin{align}\label{tilde RT}
\tilde{R}_T\leq (7\mu)^{\frac{1}{\beta}},
\end{align}
then
\begin{align}\label{rho u3}
\sup_{0\leq t\leq T}\int_\Omega\rho |\mathbf{u}|^3dx\leq \int_\Omega\rho_0|\mathbf{u}_0|^3dx+3\sqrt{2}7^{\frac{\gamma}{\beta}}\mu^{\frac{\gamma}{\beta}-1} RE_0.
\end{align}
\end{prop}

\begin{proof}
Multiplying equation $\eqref{Equ1}_2$ by $3\mathbf{u}|\mathbf{u}|$ gives
\begin{align*}
\begin{split}
&\quad\frac{d}{dt}\int_\Omega\rho |\mathbf{u}|^3dx+3\mu\int_\Omega\left(|\mathbf{u}||\nabla|\mathbf{u}||^2+|\mathbf{u}||\nabla \mathbf{u}|^2\right)dx+3\int_{\Omega}(\mu+\lambda)|\mathbf{u}||\div \mathbf{u}|^2dx\\
&\leq3\int_{\Omega}P\div(|\mathbf{u}|\mathbf{u})dx+ 3\int_{\Omega}(\mu+\lambda)|\mathbf{u}||\div\mathbf{u}||\nabla|\mathbf{u}||dx\\
&\leq 3\tilde{R}_T^\gamma\int_\Omega |\mathbf{u}|(|\div\mathbf{u}|+|\nabla|\mathbf{u}||)dx+3\int_\Omega(\mu+\lambda)|\mathbf{u}||\div\mathbf{u}|^2dx+\frac{3}{4}\int_\Omega(\mu+\lambda)|\mathbf{u}||\nabla|\mathbf{u}||^2dx.
\end{split}
\end{align*} 
Note that \eqref{tilde RT} shows that
\begin{align*}
\frac{3}{4}(\mu+\lambda)\leq 6\mu.
\end{align*}
Therefore, direct calculations imply that
\begin{align*}
\frac{d}{dt}\int_\Omega\rho |\mathbf{u}|^3dx&\leq 3\tilde{R}_T^\gamma\int_\Omega |\mathbf{u}|(|\div\mathbf{u}|+|\nabla|\mathbf{u}||)dx\\
&\leq 6(\pi R^2)^{\frac{1}{2}}\tilde{R}_T^\gamma\norm{\mathbf{u}}_{L^\infty}\norm{\nabla \mathbf{u}}_{L^2}\\
&\leq 3\sqrt{2}R\tilde{R}_T^\gamma\norm{\nabla \mathbf{u}}_{L^2}^2\\
&\leq 3\sqrt{2}(7\mu)^{\frac{\gamma}{\beta}}R\norm{\nabla \mathbf{u}}_{L^2}^2,
\end{align*}
where we used \eqref{u infty} and the following simple fact 
\begin{align*}
|\nabla|\mathbf{u}||\leq |\nabla \mathbf{u}|,\quad \norm{\div \mathbf{u}}_{L^2}\leq \norm{\nabla \mathbf{u}}_{L^2}.
\end{align*}
Integrating with respect to time, we obtain from the energy equality \eqref{energy estimate},
\begin{align*}
\sup_{0\leq t\leq T}\int_\Omega\rho |\mathbf{u}|^3dx\leq \int_\Omega\rho_0|\mathbf{u}_0|^3dx+3\sqrt{2}7^{\frac{\gamma}{\beta}}\mu^{\frac{\gamma}{\beta}-1} RE_0.
\end{align*}
\end{proof}
With the help of the Proposition \ref{u high2}, we have the following conclusion.
\begin{prop}\label{Tilda R_T}
Under the assumptions of Theorem \ref{Thm3}, we have
\begin{align}\label{tilde RT2}
\tilde{R}_T\leq \frac{99}{100}(7\mu)^{\frac{1}{\beta}}.
\end{align}
\end{prop}
\begin{proof}
Recall the transport structure
\begin{align}
\frac{D}{Dt}(\theta+\xi)+\int_R^r\frac{\rho }{s}(u_r^2-u_\theta^2)ds+P-\bar{P}+G(R,t)=0.
\end{align}
Similarly to Proposition \ref{uniform bound1}, we define
\begin{align*}
y(t)&=\theta(t,x(t)),\quad g(y)=-P(\theta^{-1}(y)),\\
h(t)&=-\xi(t,x(t))+\int_0^t\int_r^R\frac{\rho }{s}(u_r^2-u_\theta^2)dsd\tau+\int_0^t\bar{P}d\tau-\int_0^tG(R,\tau)d\tau.
\end{align*}
Therefore, for any $0<t_1<t_2<T$, we have
\begin{align*}
|h(t_2)-h(t_1)|&\leq 2\norm{\xi}_{L^\infty L^\infty}+\int_{t_1}^{t_2}\int_r^R\frac{\rho }{s}(u_r^2+u_\theta^2)dsd\tau+\int_{t_1}^{t_2}\bar{P}d\tau+\left|\int_{t_1}^{t_2}G(R,\tau)d\tau\right|\\
&:=\sum_{i=1}^{4}L_i.
\end{align*}
By using \eqref{rho u3}, one has
\begin{align*}
\norm{\xi(t)}_{L^\infty}&\leq \int_0^R\rho |u_r|ds= \int_0^R(\rho|u_r|^3s)^{\frac{1}{3}}\rho^{\frac{2}{3}}s^{-\frac{1}{3}}ds\\
&\leq \sup_{0\leq t\leq T}\left(\int_0^R \rho|u_r|^3sds\right)^{\frac{1}{3}}\tilde{R}_T^{\frac{2}{3}}\left(\int_0^R s^{-\frac{1}{2}}ds\right)^{\frac{2}{3}}\\
&\leq (2R^{\frac{1}{2}})^{\frac{2}{3}}(2\pi)^{-\frac{1}{3}}\tilde{R}_T^{\frac{2}{3}}\sup_{0\leq t\leq T}\left(\int_\Omega\rho |\mathbf{u}|^3dx\right)^{\frac{1}{3}}\\
&\leq (2R)^{\frac{1}{3}}\pi^{-\frac{1}{3}}(7\mu)^{\frac{2}{3\beta}}\left(\int_\Omega\rho_0|\mathbf{u}_0|^3dx+3\sqrt{2}7^{\frac{\gamma}{\beta}}\mu^{\frac{\gamma}{\beta}-1} RE_0\right)^{\frac{1}{3}}.
\end{align*}
Therefore, we obtain the estimate for \(L_1\),
\begin{align}\label{L1}
L_1\leq 2^{\frac{4}{3}}R^{\frac{1}{3}}\pi^{-\frac{1}{3}}(7\mu)^{\frac{2}{3\beta}}\left(\int_\Omega\rho_0|\mathbf{u}_0|^3dx+3\sqrt{2}7^{\frac{\gamma}{\beta}}\mu^{\frac{\gamma}{\beta}-1} RE_0\right)^{\frac{1}{3}}.
\end{align}
By using the energy equality $\eqref{energy estimate}$, we obtain
\begin{align}\label{L2}
L_2\leq \tilde{R}_T\int_{t_1}^{t_2}\int_0^R\frac{u_r^2+u_\theta^2}{s}dsd\tau\leq\frac{1}{2\pi}\tilde{R}_T \int_0^T\norm{\nabla \mathbf{u}}_{L^2}^2d\tau\leq \frac{1}{2 \mu\pi}(7\mu)^{\frac{1}{\beta}}E_0,
\end{align}
where we used the following simple fact
\begin{align*}
|\nabla \mathbf{u}|^2=(\partial_r u_r)^2+(\partial_r u_\theta)^2+\frac{u_r^2}{r^2}+\frac{u_\theta^2}{r^2}.
\end{align*}
The energy inequality \eqref{energy estimate} shows that
\begin{align}\label{L4}
L_3\leq \frac{(\gamma-1)E_0}{\pi R^2}(t_2-t_1).
\end{align}
Finally, we estimate the boundary term of the effective viscous flux,
\begin{align*}
L_4&=\left|\int_{t_1}^{t_2}G(R,\tau)d\tau\right|\\
&\leq\frac{1}{R^2}\left|\int_0^R\rho u_r r^2dr(t_2)-\int_0^R\rho u_r r^2dr(t_1)\right|+\left|\frac{1}{\pi R^2}\int_{t_1}^{t_2}\int_\Omega \rho^\beta \div\mathbf{u}dxd\tau\right|\\
&\quad+\left|\frac{1}{R^2}\int_{t_1}^{t_2}\int_0^R\rho (u_r^2+u_\theta^2)rdrdt\right|\\
&:=L_4^1+L_4^2+L_4^3.
\end{align*}
By mass conservation and the energy equality \eqref{energy estimate}, we obtain
\begin{align*}
L_4^1\leq \frac{2}{R}\sup_{0\leq t\leq T}\int_0^R \frac{\rho r+\rho u_r^2r}{2}dr\leq \frac{1}{2\pi R}(M_0+2E_0).
\end{align*}
The H\"{o}lder inequality shows that
\begin{align*}
L_4^2&= \frac{1}{\pi R^2}\left|\int_{t_1}^{t_2}\int_\Omega \rho^\beta \div\mathbf{u} dxd\tau\right|=\frac{1}{\pi R^2(1-\beta)}\left|\int_{t_1}^{t_2}\frac{d}{dt}\int_\Omega\rho^\beta dxd\tau\right|\\
&\leq \frac{2}{\pi R^2(1-\beta)}\sup_{0\leq t\leq T}\norm{\rho^\beta}_{L^{1/\beta}}\norm{1}_{L^{1/(1-\beta)}}\leq \frac{2}{(\pi R^2)^\beta(1-\beta)}M_0^\beta.
\end{align*}
Direct calculations show that
\begin{align*}
L_4^3&\leq \frac{1}{ R^2}\left|\int_{t_1}^{t_2}\norm{\mathbf{u}}_{L^\infty}^2\int_0^R\rho rdrdt\right|\leq \frac{M_0}{2\pi R^2}\int_0^T\norm{\mathbf{u}}_{L^\infty}^2dt\\
&\leq \frac{M_0}{4\pi^2R^2}\int_0^T\norm{\nabla\mathbf{u}}_{L^2}^2dt\leq \frac{M_0 E_0}{4\mu\pi^2 R^2}.
\end{align*}
Therefore, we obtain the estimate for \(L_4\),
\begin{align}\label{L3}
L_4\leq \frac{1}{2\pi R}(M_0+2E_0)+\frac{2}{(\pi R^2)^\beta(1-\beta)}M_0^\beta+\frac{M_0 E_0}{4\mu\pi^2 R^2}.
\end{align}
Combining \eqref{L1}-\eqref{L3}, we calculate to obtain \(N_0\) and \(N_1\) in Lemma \ref{Zlotnik},
\begin{align*}
N_0&=2^{\frac{4}{3}}R^{\frac{1}{3}}\pi^{-\frac{1}{3}}(7\mu)^{\frac{2}{3\beta}}\left(\int_\Omega\rho_0|\mathbf{u}_0|^3dx+3\sqrt{2}7^{\frac{\gamma}{\beta}}\mu^{\frac{\gamma}{\beta}-1} RE_0\right)^{\frac{1}{3}}+\frac{1}{2 \mu\pi}(7\mu)^{\frac{1}{\beta}}E_0\\
&\quad+\frac{1}{2\pi R}(M_0+2E_0)+\frac{2}{(\pi R^2)^\beta(1-\beta)}M_0^\beta+\frac{M_0 E_0}{4\mu\pi^2 R^2},\\
N_1&=\frac{(\gamma-1)E_0}{\pi R^2}.
\end{align*}
Therefore, we use Lemma \ref{Zlotnik} to obtain
\begin{align*}
\theta(t,x(t))&\leq \max\left\{\theta(0,x_0),2\mu\log((\frac{(\gamma-1)E_0}{\pi R^2})^{\frac{1}{\gamma}})+\frac{1}{\beta}(\frac{(\gamma-1)E_0}{\pi R^2})^{\frac{\beta}{\gamma}}\right\}\\
&\quad+2^{\frac{4}{3}}R^{\frac{1}{3}}\pi^{-\frac{1}{3}}(7\mu)^{\frac{2}{3\beta}}\left(\int_\Omega\rho_0|\mathbf{u}_0|^3dx+3\sqrt{2}7^{\frac{\gamma}{\beta}}\mu^{\frac{\gamma}{\beta}-1} RE_0\right)^{\frac{1}{3}}+\frac{1}{2 \mu\pi}(7\mu)^{\frac{1}{\beta}}E_0\\
&\quad+\frac{1}{2\pi R}(M_0+2E_0)+\frac{2}{(\pi R^2)^\beta(1-\beta)}M_0^\beta+\frac{M_0 E_0}{4\mu\pi^2 R^2}.
\end{align*}
Taking the supremum in time and space yields
\begin{align*}
\begin{split}
2\mu\log \tilde{R}_T+\frac{1}{\beta}\tilde{R}_T^\beta&\leq \max\left\{2\mu\log\norm{\rho_0}_{L^\infty}+\frac{1}{\beta}\norm{\rho_0}_{L^\infty}^\beta,2\mu\log((\frac{(\gamma-1)E_0}{\pi R^2})^{\frac{1}{\gamma}})+\frac{1}{\beta}(\frac{(\gamma-1)E_0}{\pi R^2})^{\frac{\beta}{\gamma}}\right\}\\
&\quad+2^{\frac{4}{3}}R^{\frac{1}{3}}\pi^{-\frac{1}{3}}(7\mu)^{\frac{2}{3\beta}}\left(\int_\Omega\rho_0|\mathbf{u}_0|^3dx+3\sqrt{2}7^{\frac{\gamma}{\beta}}\mu^{\frac{\gamma}{\beta}-1} RE_0\right)^{\frac{1}{3}}\\
&\quad+\frac{1}{2 \mu\pi}(7\mu)^{\frac{1}{\beta}}E_0+\frac{1}{2\pi R}(M_0+2E_0)+\frac{2}{(\pi R^2)^\beta(1-\beta)}M_0^\beta+\frac{M_0 E_0}{4\mu\pi^2 R^2}.\\
\end{split}
\end{align*}
By using \eqref{u infty}, we note the following simple fact :
\begin{align*}
\begin{split}
M_0&\leq \pi R^2\norm{\rho_0}_{L^\infty},\\
E_0&=\int_\Omega\frac{\rho_0^\gamma}{\gamma-1}+\frac{1}{2}\rho_0|\mathbf{u}_0|^2dx\leq \frac{\pi R^2}{\gamma-1}\norm{\rho_0}_{L^\infty}^\gamma+\frac{R^2}{4}\norm{\rho_0}_{L^\infty}\norm{\nabla \mathbf{u}_0}_{L^2}^2,\\
\int_\Omega\rho_0|\mathbf{u}_0|^3dx&\leq\norm{\rho_0}_{L^\infty}\frac{\pi R^2}{(2\pi)^{\frac{3}{2}}}\norm{\nabla\mathbf{u}_0}_{L^2}^3,
\end{split}
\end{align*}
which, combined with \eqref{def a_0}, completes the proof.
\end{proof}

\subsection{Higher order estimates}
To ensure that the local solution can be extended to a global solution, we need to use the upper bound of the density to close the estimates for higher order derivatives. Therefore, in this subsection, we use the ideas from \cite{2016 Huang-Li-JMPA} and \cite{2023 Huang-Su-Yan-Yu} to obtain the higher order estimates.

The following Proposition is a direct corollary of \eqref{the first structure}, we omit the proof here.
\begin{prop}\label{First}
Assume that 
\begin{align}
\sup_{0\le t\le T}\norm{\rho}_{L^\infty}\leq M,
\end{align}
for some positive constant $M>0$. Then there is a positive constant $C(M)$ depending only on $M,\mu,\beta,\gamma,R$ and initial data such that 
\begin{align}\label{first}
\sup_{0\leq t\leq T}\int_\Omega |\nabla\mathbf{u}|^2 dx+\int_0^T\int_\Omega\rho |\dot{\mathbf{u}}|^2dxdt\leq C(M).
\end{align}
\end{prop}

\begin{prop}\label{Second}
Under the assumptions of Proposition \ref{First}, There exists a  positive constant $C(M)$ depending only on $M,\mu,\beta,\gamma,R$ and initial data such that 
\begin{align}\label{second}
\sup_{0\leq t\leq T}\int_\Omega \rho|\dot{\mathbf{u}}|^2dx+\int_0^T\int_\Omega|\nabla\dot{\mathbf{u}}|^2dxdt\leq C(M).
\end{align}
\end{prop}

\begin{proof}
Operating $\dot{u}^j[\partial/\partial t+\div(\mathbf{u}\cdot)]$ to $\eqref{Equ1}_2^j$ and summing over \(j\) gives
\begin{align}\label{rho dot u}
\begin{split}
&\left(\dfrac{1}{2}\int_\Omega\rho|\dot{\mathbf{u}}|^2dx\right)_t\\
&=-\int_\Omega\dot{u}^j[\partial_j P_t+\div(\partial_j P\mathbf{u})]dx+\mu\int_\Omega \dot{u}^j[\partial_t\Delta u^j+\div(\mathbf{u}\Delta u^j)]dx\\
&\quad+\int_\Omega \dot{u}^j[\partial_{jt}((\mu+\lambda)\div \mathbf{u})+\div (\mathbf{u}\partial_j((\mu+\lambda)\div \mathbf{u}))]dx\\
&:=\sum_{i=1}^3O_i.
\end{split}
\end{align}
Direct integration by parts and the Young's inequality show that
\begin{align}\label{O1}
\begin{split}
O_1&=-\int_\Omega \dot{u}^j[\partial_j P_t+\div(\partial_j P\mathbf{u})]dx\\
&=\int_\Omega\left[ -P'\rho\div \mathbf{u}\partial_j\dot{u}^j+\partial_k(\partial_j\dot{u}^ju^k)P-P\partial_j(\partial_k\dot{u}^ju^k)  \right]dx\\
&\leq C\norm{\nabla \mathbf{u}}_{L^2}\norm{\nabla \dot{\mathbf{u}}}_{L^2}\\
&\leq \frac{\mu}{8}\norm{\nabla \dot{\mathbf{u}}}_{L^2}^2+C\norm{\nabla \mathbf{u}}_{L^2}^2.
\end{split}
\end{align}
Similarly,
\begin{align}\label{O2}
\begin{split}
O_2&=\mu\int_\Omega \dot{u}^j[\partial_t\Delta u^j+\div(\mathbf{u}\Delta u^j)]dx\\
&=-\mu\int_\Omega(|\nabla\dot{\mathbf{u}}|^2+\partial_i\dot{u}^j\partial_ku^k\partial_iu^j-\partial_i\dot{u}^j\partial_iu^k\partial_ku^j-\partial_iu^j\partial_iu^k\partial_k\dot{u}^j)dx\\
&\leq -\dfrac{3\mu}{4}\int_\Omega|\nabla\dot{\mathbf{u}}|^2dx+C\int_\Omega|\nabla \mathbf{u}|^4 dx.
\end{split}
\end{align}
Finally, integration by parts lead to
\begin{align}\label{O3}
\begin{split}
O_3&=\int_\Omega \dot{u}^j[\partial_{jt}((\mu+\lambda)\div \mathbf{u})+\div (\mathbf{u}\partial_j((\mu+\lambda)\div \mathbf{u}))]dx\\
&=-\int_\Omega  \partial_j\dot{u}^j[((\mu+\lambda)\div \mathbf{u})_t+\div(\mathbf{u}(\mu+\lambda)\div \mathbf{u}))]dx\\
&\quad -\int_\Omega  \dot{u}^j\div(\partial_j\mathbf{u}(\mu+\lambda)\div \mathbf{u}))dx\\
&=-\int_\Omega \left(\frac{D}{Dt}\div \mathbf{u}+\partial_j u^i\partial_i u^j\right)\left[(\mu+\lambda)\frac{D}{Dt}\div \mathbf{u}+(\mu+(1-\beta)\lambda)(\div\mathbf{u})^2\right]dx\\
&\quad +\int_\Omega  \nabla\dot{u}^j\cdot\partial_j\mathbf{u}(\mu+\lambda)\div \mathbf{u}dx\\
&\leq -\dfrac{\mu}{2}\int_\Omega \left(\frac{D}{Dt}\div \mathbf{u}\right)^2dx+\dfrac{\mu}{8}\norm{\nabla \dot{\mathbf{u}}}_{L^2}^2+C\norm{\nabla \mathbf{u}}_{L^4}^4,
\end{split}
\end{align}
where we used the simple fact:
\begin{align*}
\begin{split}
&((\mu+\lambda)\div \mathbf{u})_t+\div(\mathbf{u}(\mu+\lambda)\div \mathbf{u}))\\
&=(\mu+\lambda)\div \mathbf{u}_t+(\mu+\lambda)(\mathbf{u}\cdot \nabla)\div \mathbf{u}+\lambda_t \div \mathbf{u}+(\mathbf{u}\cdot\nabla \lambda)\div \mathbf{u}+(\mu+\lambda)(\div \mathbf{u})^2\\
&=(\mu+\lambda)\frac{D}{Dt}\div \mathbf{u}+(\mu+(1-\beta)\lambda)(\div \mathbf{u})^2.
\end{split}
\end{align*}
Substituting the estimates for \(O_1\) to \(O_3\) into \eqref{rho dot u} gives
\begin{align}\label{rho dot u2}
\begin{split}
&\left(\int_\Omega\rho|\dot{\mathbf{u}}|^2dx\right)_t+\int_\Omega|\nabla\dot{\mathbf{u}}|^2dx+\int_\Omega\left(\frac{D}{Dt}\div \mathbf{u}\right)^2dx\\
&\leq C\norm{\nabla \mathbf{u}}_{L^4}^4+C\norm{\nabla \mathbf{u}}_{L^2}^2\\
&\leq C\norm{\div \mathbf{u}}_{L^4}^4+C\norm{w}_{L^4}^4+C\norm{\nabla \mathbf{u}}_{L^2}^2\\
&\leq C(\norm{G}_{L^4}^4+\norm{P-\bar{P}}_{L^4}^2+\norm{w}_{L^4}^4+\norm{\nabla \mathbf{u}}_{L^2}^2)\\
&\leq C(\norm{G}_{L^2}^2\norm{ G}_{H^1}^2+\norm{w}_{L^2}^2\norm{\nabla w}_{L^2}^2+\norm{P-\bar{P}}_{L^4}^2+\norm{\nabla \mathbf{u}}_{L^2}^2)\\
&\leq C\Vert \sqrt{\rho}\dot{\mathbf{u}}\Vert_{L^2}^2+C\norm{\nabla \mathbf{u}}_{L^2}^2,
\end{split}
\end{align}
where we used $\eqref{first}$ and
\begin{align*}
\norm{P-\bar{P}}_{L^4}^2&\leq C\norm{P-\bar{P}}_{L^2}^2\leq C(\norm{\nabla G}_{L^2}^2+|\bar{G}|^2+\norm{(2\mu+\lambda)\div\mathbf{u}}_{L^2}^2)\\
&\leq C(\norm{\sqrt{\rho}\dot{\mathbf{u}}}_{L^2}^2+\norm{\nabla\mathbf{u}}_{L^2}^2).
\end{align*}
Integrating $\eqref{rho dot u2}$ with respect to time gives $\eqref{second}$.
\end{proof}

Finally, we estimate $\norm{\nabla \rho}_{L^q}$ and $\norm{\nabla^2 \mathbf{u}}_{L^2}$.
\begin{prop}\label{Third}
For any $q>2$, there exists a constant $C$ depending on $q,\mu,\beta,\gamma,R,T$ and initial data such that
\begin{align}\label{third}
&\sup_{0\leq t\leq T}(\norm{\rho}_{W^{1,q}}+\norm{\mathbf{u}}_{H^2})\leq C.
\end{align}
\end{prop}

\begin{proof}
Applying \(\nabla \) to the equation $\eqref{Equ1}$ gives
\begin{align}
\begin{split}
(\nabla\rho)_t+\mathbf{u}\cdot\nabla^2\rho +\nabla\mathbf{u}\cdot\nabla\rho+\nabla\rho\div\mathbf{u}+\rho\nabla\div\mathbf{u}=0.
\end{split}
\end{align}
Multiplying the above equation by \(q|\nabla\rho|^{q-2}\nabla\rho\) and integrating over $\Omega$ gives
\begin{align*}
\dfrac{d}{dt}\int_\Omega|\nabla\rho|^qdx&\leq C\int_{\Omega}|\nabla\rho|^q|\nabla\mathbf{u}|dx+C\int_\Omega|\nabla\rho|^{q-1}|\nabla\div\mathbf{u}|dx\\
&\leq C\norm{\nabla\mathbf{u}}_{L^\infty}\norm{\nabla\rho}_{L^q}^q+C\norm{\nabla\div\mathbf{u}}_{L^q}\norm{\nabla\rho}_{L^q}^{q-1}.
\end{align*}
Therefore,
\begin{align}\label{nabla rho q}
\begin{split}
\frac{d}{dt}\norm{\nabla\rho}_{L^q}&\leq C\norm{\nabla\mathbf{u}}_{L^\infty}\norm{\nabla\rho}_{L^q}+C\norm{\nabla\div\mathbf{u}}_{L^q}.
\end{split}
\end{align}
By using the Lemma \ref{B-K-M}, \eqref{first} and \eqref{GN2}, we get
\begin{align}\label{nabla u inf}
\begin{split}
\norm{\nabla \mathbf{u}}_{L^\infty}&\leq C(\norm{\div \mathbf{u}}_{L^\infty}+\norm{w}_{L^\infty})\log(e+\norm{\nabla^2 \mathbf{u}}_{L^q})+C\norm{\nabla \mathbf{u}}_{L^2}+C\\
&\leq C(\norm{G}_{L^\infty}+\norm{w}_{L^\infty}+1)\log(e+\norm{\rho\dot{\mathbf{u}}}_{L^q}+\norm{\nabla\rho}_{L^q})+C\\
&\leq C(\norm{G}_{L^2}^{\frac{q-2}{2q-2}}\norm{G}_{W^{1,q}}^{\frac{q}{2q-2}}+\norm{w}_{L^2}^{\frac{q-2}{2q-2}}\norm{\nabla w}_{L^q}^{\frac{q}{2q-2}}+1)\log(e+\norm{\rho\dot{\mathbf{u}}}_{L^q}+\norm{\nabla\rho}_{L^q})+C\\
&\leq C(\norm{\rho\dot{\mathbf{u}}}_{L^q}^{\frac{q}{2q-2}}+1)\log(e+\norm{\rho\dot{\mathbf{u}}}_{L^q}+\norm{\nabla\rho}_{L^q})+C\\
&\leq C(\norm{\rho\dot{\mathbf{u}}}_{L^q}+1)\log(e+\norm{\nabla\rho}_{L^q})+C(\norm{\rho\dot{\mathbf{u}}}_{L^q}+1),
\end{split}
\end{align}
where we used
\begin{align}\label{nabla2 u Lq}
\begin{split}
\norm{\nabla^2\mathbf{u}}_{L^q}&\leq C\norm{\nabla\div \mathbf{u}}_{L^q}+C\norm{\nabla w}_{L^q}\\
&\leq C\norm{\nabla\left(\frac{G+P-\bar{P}}{2\mu+\lambda}\right)}_{L^q}+C\norm{\nabla w}_{L^q}\\
&\leq C(\norm{\nabla G}_{L^q}+\norm{\nabla w}_{L^q}+\norm{G}_{L^\infty}\norm{\nabla\rho}_{L^q}+\norm{\nabla\rho}_{L^q}) \\
&\leq C(\norm{\nabla G}_{L^q}+\norm{\nabla w}_{L^q}+\norm{G}_{L^2}^{\frac{q-2}{2q-2}}\norm{G}_{W^{1,q}}^{\frac{q}{2q-2}}\norm{\nabla\rho}_{L^q}+\norm{\nabla\rho}_{L^q}) \\
&\leq C(\norm{\rho\dot{\mathbf{u}}}_{L^q}+\norm{\rho\dot{\mathbf{u}}}_{L^q}^{q/(2(q-1))}\norm{\nabla\rho}_{L^q}+\norm{\nabla \rho}_{L^q}).
\end{split}
\end{align}
Combining \eqref{nabla rho q} and \eqref{nabla u inf}, we obtain
\begin{align}
\begin{split}
\dfrac{d}{dt}\log\log(e+\norm{\nabla\rho}_{L^q})&\leq C+C\norm{\rho\dot{\mathbf{u}}}_{L^q}\leq C+C\norm{\nabla\dot{\mathbf{u}}}_{L^2}.
\end{split}
\end{align}
Integrating with respect to time, one has
\begin{align}\label{nabla rho}
\sup_{0\leq t\leq T}\Vert \nabla \rho\Vert_{L^q}\leq C.
\end{align}
Finally, standard elliptic estimates, along with \eqref{second} and \eqref{nabla rho}, show that
\begin{align}
\begin{split}
\norm{\nabla^2\mathbf{u}}_{L^2}&\leq C\norm{\nabla\div \mathbf{u}}_{L^2}+C\norm{\nabla w}_{L^2}\\
&\leq C\norm{\nabla\left(\frac{G+P-\bar{P}}{2\mu+\lambda}\right)}_{L^2}+C\norm{\nabla w}_{L^2}\\
&\leq C+C(\norm{\nabla G}_{L^2}+\norm{\nabla \omega}_{L^2})+\norm{G|\nabla\rho|}_{L^2}+\norm{\nabla\rho}_{L^2}\\
&\leq C+C(\norm{\nabla G}_{L^2}+\norm{\nabla \omega}_{L^2})+C\norm{G}_{L^{2q/(q-2)}}\norm{\nabla \rho}_{L^q}\\
&\leq C+C\norm{\rho\dot{\mathbf{u}}}_{L^2}+C\norm{G}_{H^1}\norm{\nabla \rho}_{L^q}\\
&\leq C+C\norm{\rho\dot{\mathbf{u}}}_{L^2}\\
&\leq C.
\end{split}
\end{align}
We complete the proof.
\end{proof}

\section{Proofs of Theorem \ref{Thm1} and \ref{Thm3}}
After all preparation done, we proceed to prove Theorem \ref{Thm1} and \ref{Thm3}. 

\subsection{Proof of Theorem \ref{Thm1}}
 Since \(\rho_0\) is radially symmetric, we continuously extend \(\rho_0\) to \(B_{2R}\), and the resulting function is still denoted as \(\rho_0\). Let
\begin{align}\label{rho0 delta}
\rho_0^\delta =\rho_0*\eta_\delta+\delta,
\end{align}
where \(\eta_\delta\) is the standard mollifying kernel. Let \(\mathbf{u}_0^\delta\) is the unique solution to the following system:
\begin{align}\label{u0 delta}
\begin{dcases}
-\mu \Delta \mathbf{u}_0^\delta-\nabla((\mu+\lambda(\rho_0^\delta))\div\mathbf{u}_0^\delta)+\nabla P(\rho_0^\delta)=\sqrt{\rho_0^\delta}g,\quad \text{in }\Omega,\\
\mathbf{u}_0^\delta=0, \quad \text{on }\partial\Omega.
\end{dcases}
\end{align}
It is easy to check:
\begin{align}
\lim\limits_{\delta\rightarrow 0}\left(\norm{\rho_0^\delta-\rho_0}_{W^{1,q}}+\norm{\mathbf{u}_0^\delta-\mathbf{u}_0}_{H^1}\right)=0.
\end{align}
According to Lemma \ref{local existence}, there exists a time \(T > 0\) and a local strong solution \((\rho^\delta, \mathbf{u}^\delta)\). Assuming the maximal existence time \(T^*\) of the local strong solution is finite, then we can slightly modified the proof of Proposition 5.1 in \cite{2016 Huang-Li-JMPA} to obtain 
\begin{align}\label{rho delta L inf}
\lim\limits_{t\rightarrow T^*}\norm{\rho^\delta (t)}_{L^\infty}=\infty.
\end{align}
We obtain \eqref{rho delta L inf} by employing the proof by contradiction. Assuming that \eqref{rho delta L inf} does not hold, we claim there exists a constant \(\hat{C} > 0\), which may depend on \(T^*\) but not on \(T<T^*\), such that
\begin{align}\label{non vacuum}
\inf_{(x,t)\in \Omega\times(0,T)}\rho^\delta(x,t)\ge \hat{C}^{-1},
\end{align}
and for any $0<T<T^*$, 
\begin{align}\label{rho H^2}
\sup_{0\le t\le T}(\norm{\rho^\delta}_{H^2}+\norm{\mathbf{u}^\delta}_{H^2})\leq \hat{C}.
\end{align}

We first show that \(\rho^\delta\) is away from vacuum on $(0,T)\times\Omega$. Note that
\begin{align}
\rho^\delta(x,t)=\rho^\delta_0(x_0)\exp\left\{-\int_0^t\div \mathbf{u}^\delta(x(\tau),\tau)d\tau\right\}.
\end{align}
By standard elliptic regularity theory, for some $q>2$, we obtain
\begin{align}
\begin{split}
\norm{\div \mathbf{u}^\delta}_{L^\infty}&\leq \hat{C}\norm{\nabla \mathbf{u}^\delta}_{W^{1,q}}\leq \hat{C}+\hat{C}\norm{\nabla \div\mathbf{u}^\delta}_{L^q}+\hat{C}\norm{\nabla \omega^\delta}_{L^q}\\
&\leq \hat{C}+\hat{C}\norm{\nabla\left(\frac{G^\delta+P^\delta-\bar{P}^\delta}{2\mu+\lambda(\rho^\delta)}\right)}_{L^q}+\hat{C}\norm{\nabla w^\delta}_{L^q}\\
&\leq \hat{C}+\hat{C}\norm{G^\delta}_{W^{1,q}}+\hat{C}\norm{\nabla w^\delta}_{L^q}\\
&\leq \hat{C}+\hat{C}\norm{\nabla \dot{\mathbf{u}}^\delta}_{L^2},
\end{split}
\end{align}
where we used \eqref{third} and \eqref{G,w}. Therefore, from \eqref{second} one has 
\begin{align}
\inf_{(x,t)\in \Omega\times(0,T)}\rho^\delta(x,t)\ge\inf_{x_0\in \Omega}\rho_0^\delta(x)\exp\left\{-\int_0^T\norm{\div \mathbf{u}^\delta}_{L^\infty}d\tau\right\}\ge\hat{C}^{-1}.
\end{align}

Next, we will prove \eqref{rho H^2}. Due to \eqref{third}, we only need to prove that
\begin{align}\label{nabla2rho}
\sup_{0\le t\le T} \norm{\nabla^2 \rho^\delta(t)}_{L^2}\leq \hat{C}.
\end{align}
Applying \(\partial_i \partial_j\) to equation $\eqref{Equ1}_1$, we obtain
\begin{align}
\begin{split}
\partial_t(\partial_{ij}\rho^\delta)&+\partial_{ij} u^\delta_k\partial_k \rho^\delta+\partial_i u^\delta_k\partial_{jk}\rho^\delta+\partial_j u_k^\delta\partial_{ik}\rho^\delta+u_k^\delta \partial_{ijk}\rho^\delta\\
&+\partial_{ij}\rho^\delta\partial_ku^\delta_k+\partial_i\rho^\delta\partial_{jk}u^\delta_k+\partial_j\rho^\delta\partial_{ik}u^\delta_k+\rho^\delta \partial_{ijk} u^\delta_k=0.
\end{split}
\end{align}
Multiplying the above equation by $\partial_{ij}\rho^\delta$, summing over \(i, j\), and integrating over $\Omega$ yields
\begin{align}\label{ddt nabla2rho}
\begin{split}
&\quad\frac{d}{dt}\int_\Omega |\nabla^2\rho^\delta|^2dx\\
&\leq \hat{C} \left(\int_\Omega |\nabla\rho^\delta||\nabla^2\mathbf{u}^\delta||\nabla^2\rho^\delta|dx+\int_\Omega|\nabla \mathbf{u}^\delta||\nabla ^2\rho^\delta|^2dx+\int_\Omega\rho^\delta|\nabla^3 \mathbf{u}^\delta||\nabla^2\rho^\delta|dx\right)\\
&\leq \hat{C}(1+\norm{\nabla \mathbf{u}^\delta}_{L^\infty})\int_\Omega |\nabla^2\rho^\delta|^2dx+\hat{C}\norm{\nabla \rho^\delta}_{L^q}^2\norm{\nabla^2 \mathbf{u}^\delta}_{H^1}^2+\hat{C}\norm{\rho^\delta}_{L^\infty}^2\norm{\nabla^3 \mathbf{u}^\delta}_{L^2}^2\\
&\leq \hat{C}(1+\norm{\nabla \dot{\mathbf{u}}^\delta}_{L^2})\int_\Omega |\nabla^2\rho^\delta|^2dx+\hat{C}+\hat{C}\norm{\nabla^3 \mathbf{u}^\delta}_{L^2}^2,
\end{split}
\end{align}
where we have used $\eqref{third}.$ Next, we estimate $\norm{\nabla^3 \mathbf{u}^\delta}_{L^2}^2$. Standard elliptic theory indicates that
\begin{align}
\begin{split}
\norm{\nabla^3 \mathbf{u}^\delta}_{L^2}&\leq \hat{C}+ \hat{C}\norm{\nabla^2\div\mathbf{u}^\delta}_{L^2}+\hat{C}\norm{\nabla^2\omega^\delta}_{L^2}\\
&\leq \hat{C}+ \hat{C}\norm{\nabla^2\left(\frac{G^\delta+P^\delta-\bar{P}^\delta}{2\mu+\lambda(\rho^\delta)}\right)}_{L^2}+\hat{C}\norm{\nabla^2\omega^\delta}_{L^2}\\
&\leq \hat{C}+\hat{C}\left(\norm{\nabla^2 G^\delta}_{L^2}+\norm{\nabla G^\delta}_{L^{\frac{2q}{q-2}}}\norm{\nabla \rho^\delta}_{L^q}+\norm{G^\delta}_{L^\infty}\norm{\nabla^2\rho^\delta}_{L^2}\right.\\
&\quad \left.+\norm{G^\delta}_{L^\infty}\norm{\nabla \rho^\delta}_{L^4}^2+\norm{\nabla^2 \rho^\delta}_{L^2}+\norm{\nabla\rho^\delta}_{L^4}^2+\norm{\nabla^2\omega^\delta}_{L^2}\right)\\
&\leq \hat{C}\left(1+\norm{\nabla^2 G^\delta}_{L^2}+\norm{\nabla^2\omega^\delta}_{L^2}\right)+\hat{C}(1+\norm{\nabla \dot{\mathbf{u}}^\delta}_{L^2})(1+\norm{\nabla^2\rho^\delta}_{L^2}).
\end{split}
\end{align}
Next, we estimate the most challenging parts, $\norm{\nabla^2 G^\delta}_{L^2}$ and $\norm{\nabla^2\omega^\delta}_{L^2}$. Direct calculations show that
\begin{align*}
\partial_{ij} G^\delta&=\partial_j\left(\partial_r G^\delta\frac{x_i}{r}\right)=\partial_{rr}G^\delta\frac{x_ix_j}{r^2}+\partial_r G^\delta\frac{\delta_{ij}r^2-x_ix_j}{r^3},
\end{align*}
and 
\begin{align*}
\Delta G^\delta=\partial_{rr}G^\delta+\frac{\partial_rG^\delta}{r}.
\end{align*}
Therefore,
\begin{align}
\begin{split}
\norm{\nabla^2 G^\delta}_{L^2}&\leq \hat{C}\left(\norm{\partial_{rr} G^\delta}_{L^2}+\norm{\frac{\partial_r G^\delta}{r}}_{L^2}\right)\leq \hat{C}\left(\norm{\Delta G^\delta}_{L^2}+\norm{\frac{\partial_r G^\delta}{r}}_{L^2}\right)\\
&\leq \hat{C}\left(\norm{\nabla(\rho^\delta\dot{\mathbf{u}}^\delta)}_{L^2}+\norm{\frac{\rho^\delta \dot{\mathbf{u}}^\delta}{r}}_{L^2}\right)\leq \hat{C}\norm{\nabla \dot{\mathbf{u}}^\delta}_{L^2},
\end{split}
\end{align}
where we used 
\begin{align*}
|\nabla\dot{\mathbf{u}}^\delta|^2=(\partial_r \left<\dot{\mathbf{u}}^\delta, \mathbf{e_r}\right>)^2+(\partial_r \left<\dot{\mathbf{u}}^\delta, \mathbf{e_\theta}\right>)^2+\frac{|\left<\dot{\mathbf{u}}^\delta, \mathbf{e_r}\right>|^2}{r^2}+\frac{|\left<\dot{\mathbf{u}}^\delta, \mathbf{e_\theta}\right>|^2}{r^2}\ge\left|\frac{\dot{\mathbf{u}}^\delta}{r}\right|^2.
\end{align*}
Similarly, we obtain
\begin{align*}
\norm{\nabla^2\omega^\delta}_{L^2}\leq \hat{C}\norm{\nabla \dot{\mathbf{u}}^\delta}_{L^2}.
\end{align*}
Therefore, one has 
\begin{align}
\norm{\nabla^3 \mathbf{u}^\delta}_{L^2}\leq \hat{C}(1+\norm{\nabla \dot{\mathbf{u}}^\delta}_{L^2})\norm{\nabla^2\rho^\delta}_{L^2}+\hat{C}\norm{\nabla\dot{\mathbf{u}}^\delta}_{L^2}+\hat{C}.
\end{align}
Substituting the above estimates into $\eqref{ddt nabla2rho}$ gives 
\begin{align}
\begin{split}
\frac{d}{dt}\int_\Omega |\nabla^2\rho^\delta|^2dx\leq \hat{C}(1+\norm{\nabla \dot{\mathbf{u}}^\delta}_{L^2}^2)\int_\Omega |\nabla^2\rho^\delta|^2dx+\hat{C}\norm{\nabla\dot{\mathbf{u}}^\delta}_{L^2}^2+\hat{C}.
\end{split}
\end{align}
By using the Gronwall inequality and \eqref{second}, we complete the proof of \eqref{nabla2rho}. Then, we can define \(\rho^\delta(T^*)\) and \(\mathbf{u}^\delta(T^*)\). According to Lemma \ref{local existence}, there exists a positive time interval starting from \(T^*\) on which a local solution exists. This contradicts the definition of \(T^*\). Therefore, \eqref{rho delta L inf} holds, which contradicts \eqref{uniform bound 1}, \eqref{uniform bound2} and \eqref{upper bound density}. Therefore, we have proven that $(\rho^\delta, \mathbf{u}^\delta)$ is global.

Letting $\delta\rightarrow 0$, standard compactness arguments ensure the existence of a unique strong solution to $\eqref{Equ 2}$-$\eqref{boundary condition2}$ satisfying \eqref{regularity}. As for asymptotic behaviors \eqref{long time behaviors}, the proof is the same as in \cite{2016 Huang-Li-JMPA}, so we omit it.

\subsection{Proof of Theorem \ref{Thm3}}
Assuming \(\rho_0\) satisfies $\eqref{rho 0 small}$, let \(\delta\) be sufficiently small such that
\begin{align*}
\norm{\rho_0^\delta}_{L^\infty}\leq K_{\mu, \beta,\gamma, R, \norm{\nabla \mathbf{u}_0^\delta}_{L^2}}^{-1}\left(\theta\left(\frac{199}{200}(7\mu)^{\frac{1}{\beta}}\right)\right),
\end{align*}
where \(\rho_0^\delta\) is shown as in \eqref{rho0 delta} and $\mathbf{u}_0^\delta$ satisfy \eqref{u0 delta}. By Lemma \ref{local existence}, there exists a time \(T > 0\) and a unique strong solution \((\rho^\delta, \mathbf{u}^\delta)\) on \((0, T) \times \Omega\). Suppose that the maximal existence time \(T^*\) of the local strong solution is finite, then 
\begin{align}
\lim\limits_{t\rightarrow T^*}\norm{\rho^\delta (t)}_{L^\infty}=\infty,
\end{align}
By the continuity of the density with respect to time, we define \(\tilde{T}\) to satisfy
\begin{align*}
\norm{\rho^\delta(t)}_{L^\infty}< (7\mu)^{\frac{1}{\beta}},\quad 0<t<\tilde{T},
\end{align*}
and 
\begin{align*}
\norm{\rho(\tilde{T})}_{L^\infty}=(7\mu)^{\frac{1}{\beta}}.
\end{align*}
However, according to \eqref{tilde RT2}, such a \(\tilde{T}\) does not exist. Therefore, \(T^* = \infty\). 

Let \(\delta \to 0\), we complete the proof using the same argument as in Theorem \ref{Thm1}.
\section*{Acknowledgments}
X.-D. Huang is partially supported by CAS Project for Young Scientists in Basic Research,
Grant No. YSBR-031 and NNSFC Grant Nos. 12494542 and 11688101.

\vspace{1cm}
\noindent\textbf{Data availability statement.} Data sharing is not applicable to this article.

\vspace{0.3cm}
\noindent\textbf{Conflict of interest.} The authors declare that they have no conflict of interest.

\end{document}